\documentclass[12pt]{article}
\usepackage{graphicx} 

\usepackage{amssymb,amsmath,enumitem}
\usepackage{amsthm}
\usepackage{thmtools}
\usepackage[margin=1in]{geometry}
\usepackage{authblk}
\usepackage{lineno}
\usepackage{xcolor}
\usepackage{tikz}
\usepackage[hyphens]{xurl}
\usepackage[colorlinks=true, linkcolor = violet, citecolor = blue, urlcolor = blue]{hyperref}
\hypersetup{
	pdftitle = {Peace on Earth},
	pdfauthor = {Chun-Hung Liu  and Bruce Reed}
}
\declaretheorem[name=Lemma, numberwithin = section]{lemma}
\declaretheorem[name=Theorem,sibling = lemma]{theorem}

\declaretheorem[name=Observation, sibling=lemma]{observation}
\declaretheorem[name=Definition, sibling=lemma]{definition}
\declaretheorem[name=Corollary, sibling=lemma]{corollary}

\usepackage[noabbrev,capitalise,nameinlink]{cleveref}
\crefname{claim}{Claim}{Claims}
\crefname{lemma}{Lemma}{Lemmas}
\crefname{theorem}{Theorem}{Theorems}
\crefname{proposition}{Proposition}{Propositions}
\crefname{question}{Question}{Questions}
\crefname{conjecture}{Conjecture}{Conjectures}
\crefname{figure}{Figure}{Figures}
\crefname{corollary}{Corollary}{Corollaries} 
\crefformat{equation}{(#2#1#3)}
\Crefformat{equation}{Equation #2(#1)#3}

\renewcommand{\epsilon}{\varepsilon}

\newcommand{\pcf}{{\rm pcf}}   
\newcommand{\Med}{{\rm Med}}   
\newcommand{\dom}{{\rm dom}}

\title{Peaceful Colourings}
\author{Chun-Hung Liu\footnote{Department of Mathematics, Texas A\&M University, USA. chliu@tamu.edu. Partially supported by NSF under award DMS-1954054 and CAREER award DMS-2144042.},
\quad Bruce Reed \footnote{Institute of Mathematics, Academia Sinica, Taiwan. bruce.al.reed@gmail.com.  Supported by  NSTC Grant 112-2115-M-001 -013 -MY3}}

\begin{document}

\maketitle

\begin{abstract}
   We introduce peaceful colourings, a variant of $h$-conflict free colourings.  We call a colouring with no monochromatic edges  $p$-peaceful if for each vertex $v$, there are at most $p$ neighbours  of $v$ coloured with a colour appearing on another neighbour of $v$. 
   An $h$-conflict-free colouring of a graph is a (vertex)-colouring with no monochromatic edges so that for every vertex $v$, the number of neighbours of $v$ 
   which are coloured with a colour appearing on no other neighbour of $v$ is at least the minimum of $h$ and the degree of $v$. 
   If $G$ is $\Delta$-regular then it has an $h$-conflict free colouring precisely if it has a $(\Delta-h)$-peaceful colouring. 
   We focus on the minimum $p_\Delta$ of those $p$  for which every graph of maximum degree $\Delta$ has a $p$-peaceful colouring with $\Delta+1$ colours.  
   We show that $p_\Delta > (1-\frac{1}{e}-o(1))\Delta$ and that for graphs of bounded codegree, $p_\Delta \leq  (1-\frac{1}{e}+o(1))\Delta$. 
   We ask if the latter result can be improved by dropping the bound on the codegree. 
   As a partial result, we show that $p_\Delta \leq \frac{8000}{8001}\Delta$ for sufficiently large $\Delta$. 
\end{abstract}

\section{Introduction}

Many graph invariants are defined as  the number of colours required to colour the vertices of a graph so that 
its edges are non-monochromatic and some other property holds with respect to the colours used on each neighbourhood. 
For example if we wish to colour the square of the graph we are requiring that every colour be used at most 
once in the neighbourhood. 
In an $b$-frugal colouring \cite{hmr,mr} we require that every colour be used at most $b$ times. 
It is also natural to bound the number of exceptional vertices which use a colour appearing more than $b$ times. 
In this vein, for  a real number $p$, we define a {\it $p$-peaceful colouring} to be a colouring $f$ with
no monochromatic edges in which  for every vertex $v$:  
$$|\{w \in N(v) : \exists u \in N(v)-\{w\} \text{ with } f(u)=f(w)\}| \le p.$$

Peaceful colourings are related to the more studied but perhaps less natural conflict-free colourings. 
A {\it proper conflict-free colouring} of a graph is a (vertex-)colouring with no monochromatic edges such that for every non-isolated vertex $v$, the neighbourhood $N(v)$ contains a vertex $w$ coloured with a colour not appearing on $N(v)-\{w\}$. 
This concept was recently introduced by Fabrici, Lu\v{z}ar, Rindo\v{s}ova and Sot\'{a}k \cite{flrs}. 
It has received a considerable amount of attention \cite{alo,accchhhkz,cps,cckp,cckp2,cl,h,kp,l,lr}.
Caro et al.\ \cite{cps} conjectured that if $\Delta \ge 3$ then every connected graph of maximum degree $\Delta$ has a proper conflict-free $(\Delta+1)$-colouring.
This conjecture remains open for $\Delta \geq 4$, while Liu and Yu \cite{ly} proved the case $\Delta=3$.
Liu and Reed \cite{lr} proved that $\Delta + O(\Delta^{2/3}\log\Delta)$ colours suffice for sufficiently large $\Delta$.

Cho, Choi, Kwon and Park introduced $h$-conflict-free proper colourings in \cite{cckp2}. 
For a real number $h$, a colouring of a graph  with no monochromatic edges is {\it $h$-conflict-free} if for every vertex $v$, $N(v)$ contains at least $\min\{\deg(v),h\}$ vertices coloured with a colour used only once in $N(v)$; we use $\chi^h_{\pcf}(G)$ to denote the fewest number of colours in such a colouring of $G$. 

Kamyczura and Przybylo proved the following in \cite{kp}:

\begin{theorem}[\cite{kp}]
    \label{kptheorem}
    There is a $\Delta_0$ such that if $G$ has maximum degree $\Delta  \ge \Delta_0$,
    and minimum degree $\delta \ge 1500\log\Delta$, then for any $h \le \frac{\delta}{75}$, $\chi_{\pcf}^h(G) \le \Delta(1+\max\{\frac{30h}{\delta}, \frac{600 \log \Delta}{\delta}\})$. 
\end{theorem}

This implies the following: 

\begin{corollary}
    \label{kpcor}
    There is a $\Delta_0$ such that if $G$ is a $\Delta$-regular graph of maximum 
    degree $\Delta  \ge \Delta_0$, then for any $h \le \frac{\Delta}{75}$, $\chi_{\pcf}^h(G) \le \Delta+\max\{30h, 600 \ln \Delta\}$. 
\end{corollary}

One consequence of our results (Theorem \ref{lowerbound}) is a strengthening of this corollary for small $h$ and proves the aforementioned conjecture of Caro et al.\ \cite{cps} for regular graphs of large degree in a strong form: for large enough $\Delta$, every $\Delta$-regular graph $G$ satisfies $\chi_{pcf}^h(G) \le \Delta+1$ for any $h \le \frac{\Delta}{8001}$. 
This is best possible for such $h$ as a $(\Delta+1)$-clique shows.

We note that for a $d$-regular graph, a colouring is an $h$-conflict-free proper colouring precisely if it 
is a $(d-h)$-peaceful colouring. In contrast, if  $G$ is an irregular graph of maximum degree $\Delta$ 
then while a $p$-peaceful colouring and a $(\Delta-p)$-conflict-free colouring impose the same condition
on maximum degree vertices, the peaceful colouring imposes weaker conditions on low degree vertices. 

Note that every graph $G$ of maximum degree $\Delta$ has a proper $(\Delta+1)$-colouring which is $\Delta$-peaceful. 
On the other hand, a 1-peaceful colouring of $G$ is a colouring of its square and this may require $\Omega(\Delta^2)$ colours. 

In this paper, we consider how peaceful a proper $(\Delta+1)$-colouring of a graph of maximum degree $\Delta$ can be.
In particular, we define $f(\Delta)$ to be the minimum $p$ such that every graph of maximum degree $\Delta$ permits a $p$-peaceful colouring using $\Delta+1$ colours.

We obtain the following lower bound for $f(\Delta)$, even for graphs of low codegree: 
(The {\it codegree} of two distinct vertices is the number of their common neighbours, and the {\it maximum codegree} of a graph is the maximum of the codegree of any two distinct vertices of the graph.)

\begin{theorem}
\label{upperbound}
    There is a $\Delta_0$ such that for $\Delta \ge \Delta_0$, there are bipartite graphs of maximum degree $\Delta$ and maximum codegree at most $\log^2 \Delta$ which have no $((1-e^{-1})\Delta-\frac{\Delta}{\log^{1/9}\Delta})$-peaceful colourings with $\Delta+1$ colours. 
    Thus for sufficiently large $\Delta$, $$f(\Delta) > (1-e^{-1})\Delta-\frac{\Delta}{\log^{1/9}\Delta}.$$ 
\end{theorem}

This bound is essentially  tight for  graphs of low codegree.   

\begin{theorem}
\label{sparsegraphs}
     There is a $\Delta_0$ and a $C$     
     such that for $\Delta \ge \Delta_0$, every graph of  maximum degree $\Delta$ and  codegree at most $\frac{\sqrt{\Delta}}{\log^8 \Delta}$ has a proper $(\Delta+1)$-colouring which is $\lceil (1-e^{-1})\Delta+ \frac{C\Delta}{\log \log \Delta} \rceil  $ peaceful.  
\end{theorem}

We ask whether Theorem \ref{sparsegraphs} can be strengthened by dropping the condition that the codegree is bounded.
In this direction, as mentioned earlier, we show that $\Delta-f(\Delta)$ is $\Omega(\Delta)$.

\begin{theorem}
\label{lowerbound}
    There is a $\Delta_0$ such that for $\Delta \ge \Delta_0$, every graph of maximum degree $\Delta$ has a proper $(\Delta+1)$-colouring which is $\frac{8000}{8001}\Delta$-peaceful.
    Thus, for sufficiently large $\Delta$, $$f(\Delta)  \le \frac{8000}{8001}\Delta.$$ 
\end{theorem}

Recall that Caro et al.\ \cite{cps} conjectured that every connected graph of maximum degree $\Delta \geq 3$ has a proper conflict-free coloring with $\Delta+1$ colors. 
Theorem \ref{lowerbound} proves this conjecture for regular graphs of large degree in a strong form.

We can embed a graph $G$  of maximum degree $\Delta$ in a $\Delta$-regular graph by repeatedly taking two copies of 
the current host graph and adding an edge between the two copies of any vertex of degree less than $\Delta$.
Now, if the host graph has a $p$-peaceful colouring so does $G$. So, $f(\Delta)$ is the maximum over all 
$\Delta$-regular graphs $F$ of the difference between $\Delta$ and an integer $h$ for which $\chi_\pcf^h(F) \leq \Delta+1$. 
Hence, to prove Theorem \ref{lowerbound}, we need only prove it for $\Delta$-regular graphs.  The same is true for 
Theorem \ref{sparsegraphs}, as the maximum codegree between two vertices in different copies of the graph we 
duplicate in the construction process is at most two.

We observe that if we allow a number of colours which is linear in $\Delta$ then we can achieve almost perfect peace.

\begin{theorem}
\label{manycolours} 
     For every $\mu>0$, there is a $\Delta_\mu$, such that for $\Delta \ge \Delta_\mu$, every graph of maximum degree at most $\Delta$ has a proper $\lfloor 5\mu^{-1}\Delta \rfloor$-colouring which is $\mu \Delta$-peaceful.  
\end{theorem}

This paper is organized as follows.
Section \ref{sec:preleminaries} presents some standard probabilistic tools and a few simple observations we will need. 
We then present  the straightforward proof of Theorem \ref{manycolours} in Section \ref{sec:proof_manycolours}. It illustrates the  main approach we use  in a very simple setting.   
We next prove Theorem \ref{upperbound} in Section \ref{sec:proof_upperbound}, which, as it is a lower bound, requires a different approach than the rest of the results.
In Section \ref{sec:proof_lowerbound}, we combine the ideas in the proof of Theorem \ref{manycolours} with further machinery, including decomposing graphs into small dense sets and a set of vertices  with sparse neighbourhoods, to prove Theorem \ref{lowerbound}.
Finally, we prove Theorem \ref{sparsegraphs}, using an iterative pseudo-random procedure.

\section{Preliminaries} \label{sec:preleminaries}

Let $G$ be a graph.
A {\it partial colouring} of $G$ is a function whose domain is a subset of $V(G)$ such that there is no monochromatic edge. For a function $f$, we let $dom(f)$ be the domain of $f$. 
For a partial colouring $f$ of $G$, we say that a neighbour  $w$ of $v$  is {\it $(f,v)$-undisturbed} if $w$ is coloured and there is no $u \in N(v)-w$ with $f(u)=f(w)$. We let  $U_v$ be the number of $(f,v)$-undisturbed vertices.
Note that a proper colouring $f$ of $G$ is $p$-peaceful if and only if every $v \in V(G)$ satisfies $U_v \ge \deg(v)-p$.

The following straightforward observation will be used frequently, we prove it by simply completing 
the colouring greedily. 

\begin{observation} \label{greedy_obs}
    Let $p$ be a nonnegative real number.
    Let $G$ be a graph with maximum degree at most $\Delta$.
    If for some $c \ge \Delta+1$, there is a partial $c$-colouring $f$ of $G$ such that for every vertex $v \in V(G)$ which is not in a clique of size at least $\Delta+1-\frac{p}{2}$, $U_v-|N(v)-\dom(f)| \geq \deg(v)-p$,  then $G$ has a $p$-peaceful $c$-colouring. 
\end{observation}

\begin{proof}
Since $c \geq \Delta+1$, we can extend $f$ to a proper $c$-colouring $f'$ of $G$.
For every vertex $v$  of $G$ in a clique of size $\Delta+1-\frac{p}{2}$, its neighbours in the clique  receive distinct colours 
and at most $\frac{p}{2}$ receive a colour  assigned to a neighbour of $v$  not in the clique. 
For every  other vertex $v$, there are at most  $|N(v)-\dom(f)|$ colours which are used by $f'$ on $N(v)$ but not by $f$. So there are at least $U_v-|N(v)-\dom(f)| \geq \deg(v)-p$ $(f',v)$-undisturbed vertices.
So $f'$ is a $p$-peaceful colouring.
\end{proof}

We need the following which allows us to analyze the neighbourhood of each vertex separately.

\begin{lemma}[Lov\'{a}sz Local Lemma \cite{el}] \label{lll}
Let $p<1$ be a nonnegative real number.
Let $d$ be a nonnegative integer.
Let ${\mathcal E}$ be a set of events such that for every $A \in {\mathcal E}$, $P(A) \leq p<1$ and there is a set $S_A$ of  at most $d$ events such that $A$ is mutually independent of ${\mathcal E}-S_A$.
If $4pd<1$, then with positive probability, none of the events in ${\mathcal E}$ occurs.
\end{lemma}

We also need the following bespoke theorem which Colin McDiarmid \cite{mc}  developed specifically so it could be 
applied to quasi-random colouring procedures such as the one  we discuss. We present a slightly simplified version.  
We use $\Med(X)$ to denote the median of a random variable $X$. 

\begin{theorem}[\cite{mc}]
\label{concentrationtheorem}
Suppose that $X$ is a nonnegative valued random variable on the elements of a product space each component of which is a uniformly  random permutation on a finite set.  
Suppose further that there are real numbers $c$ and $r$ such that swapping the position of any two elements in a permutation can change the value of $X$ by at most $c$, and that for every real number $s \geq 0$, to certify $X \geq s$, we need only specify the elements in $rs$ positions where these positions can be in any of the permutations. 
Then
\begin{enumerate}
    \item  for all positive $t$, $P(X>{\rm Med}(X)+t) \le 2e^{\frac{-t^2}{4rc^2({\rm Med}(x)+t)}}$, and  
    \item  for all positive $t$, $P(X<{\rm Med}(X)-t) \le 2e^{\frac{-t^2}{4rc^2{\rm Med}(X)}}$. 
\end{enumerate}
\end{theorem}

We note that we can and do  think of a choice of an element from a finite set as a 
permutation of the set where we choose the first element of the permutation. 
We note further that McDiarmid stated the lemma with the hypothesis  that swapping two elements 
changed the outcome by at most $2c$ and hence had a $16$ where we have a four. 

This inequality bounds concentration around the median not the mean, so we need to show that these two are close.
Using Lemma 4.6 in \cite{MC}, we obtain the following. 

\begin{corollary}[See \cite{lr}]
\label{permutation}
    If the hypotheses of Theorem \ref{concentrationtheorem} hold, then for all  $t>0$, $$P \left (|X-E[X]|>t+ 6c\sqrt{2rE[X]}+16c^2r \right ) \le 4e^{\frac{-t^2}{8rc^2(E[X]+t)}}.$$ 
\end{corollary}

We also need the following: 

\begin{theorem}[Azuma's Inequality \cite{a}] \label{azuma}
Let $X$ be a random variable determined by $n$ trials $T_1,T_2,...,T_n$ such that for each $i$, and any two sequences of outcomes $t_1,t_2,...,t_i$ and $t_1,t_2,...,t_{i-1},t_i'$, 
$$\Bigg|E \bigl[ X|T_1=t_1,T_2=t_2,...,T_i=t_i \bigr] - E\bigl[ X|T_1=t_1,T_2=t_2,...,T_{i-1}=t_{i-1},T_i=t_i' \bigr] \Bigg| \leq c_i.$$
Then $P(|X-E[X]|>t) \leq 2e^{-\frac{t^2}{2\sum_{i=1}^nc_i^2}}$.
\end{theorem}

\section{The Proof of Theorem \ref{manycolours}} \label{sec:proof_manycolours}

We prove Theorem \ref{manycolours} in this section.

We set  $c=\lfloor 5\mu^{-1} \Delta \rfloor$ and consider the random process where we assign a random colour from $\{1,...,c\}$ to every vertex of $G$.
Let $g$ be the resulting $c$-colouring of $G$.
Then we uncolour each vertex which has the same colour as a neighbour. 
Let $f$ be the resulting partial $c$-colouring of $G$.

For every $v \in V(G)$, let $A_v$ be the event that the number of $(f,v)$-undisturbed vertices is strictly smaller than $\deg(v)-\mu\Delta+|N(v)-\dom(f)|$.
We note that this event $A_v$ is determined by vertices at distance at most two from $v$.
Hence $A_v$ is mutually independent of $\{A_w:w \in V(G)\}-{\cal S}_v$, where ${\cal S}_v$ consists of the events $A_w$ for vertices $w$ within distance at most four from $v$.
So if each $A_v$ has probability at most $\frac{1}{4\Delta^4}$ then we can apply the Local Lemma (Theorem \ref{lll}) to show that with positive probability none of our events occur and hence $G$ has a $\mu\Delta$-peaceful colouring by Observation \ref{greedy_obs}.

Let $v \in V(G)$.

Let $X$ be the number of those $w$ in $N(v)$ which are   assigned a colour by $g$ which is  
assigned to  another element of $N(v) \cup N(w)$.  
Note that the number of $(f,v)$-undisturbed vertices equals $\deg(v)-X$. 
Furthermore $X \ge |N(v)-\dom(f)|$.
If $A_v$ holds then  $\deg(v)-X < \deg(v)-\mu\Delta+X$. 
That is, $P(A_v) \leq P(\deg(v)-X < \deg(v)-\mu\Delta+X)$.
So it suffices to show that $P(X \ge \frac{\mu \Delta}{2}) \le  \frac{1}{4\Delta^4}$.
We show that this is true conditioned on any colouring of $V(G)-N(v)$, which yields the desired 
result. 

The probability that a neighbour $w$ of $v$ is assigned a colour assigned to another vertex
of $N(v) \cup N(w)$ given our conditioning is at most $\frac{|N(v) \cup N(w)|}{c}$. Since, $c>\frac{9\Delta}{2\mu}$,
this is less than $\frac{4\mu}{9}$. Hence $E[X] \le \frac{4\mu \deg(v)}{9} \le \frac{4\mu \Delta}{9}$.
So it suffices to show $P(|X-E[X]| > \frac{\mu \Delta}{18})<\frac{1}{4\Delta^4}$.

Note that  $X$ is determined by the trials that decide $(g(u): u \in N(v))$.
Changing the outcome of each trial can affect $X$ by at most 2, and to certify $X \geq s$, we only need to specify $2s$ trials. Now, for sufficiently large $\Delta$, $12\sqrt{4\frac{4\mu\Delta}{9}} +128<\frac{\mu\Delta}{36}$. 
So by taking $r=c=2$ in Corollary \ref{permutation}, for large $\Delta$ we have: 
    \begin{align*}
        P\left(|X-E[X]|>\frac{\mu\Delta}{18}\right) \leq & P\left(|X-E[X]|>\frac{\mu\Delta}{36}  + 12\sqrt{4E[X]} +128\right) \\
        \leq & 4e^{-\frac{(\mu\Delta/36)^2}{64(E[X]+(\mu\Delta/36)}} 
        \leq  4e^{-\frac{(\mu\Delta/36)^2}{64\cdot 17 \mu \Delta/36}}  \leq \frac{1}{8\Delta^4}.
    \end{align*}
This proves the theorem.

\section{The Proof of Theorem \ref{upperbound}} \label{sec:proof_upperbound}

We prove Theorem \ref{upperbound} in this section.
The key to the proof is the following:

\begin{lemma} \label{lem_upperbound_bipartite}
   There is a $\Delta_0$ such that for $\Delta \ge \Delta_0$ there is a graph $G$ with bipartition $(A,B)$ such that $\Delta(G)=\Delta$, the  maximum codegree of $G$ is at most  $\log^2 \Delta$, $|A|=\lfloor \frac{\Delta^2}{\log^{3/2} \Delta} \rfloor$, $|B|=\Delta$, every vertex of $B$ is adjacent to $\Delta$ vertices of $A$, and for every $S \subseteq A$ with $|S| \le \frac{\Delta}{\log^{5/4}\Delta}$, the number of vertices of $B$ which are adjacent to exactly one vertex of $S$ is at most $(e^{-1}+\frac{1}{3\log^{1/9}\Delta})|B|$.  
\end{lemma}

\begin{proof}    
Let $A$ and $B$ be two disjoint sets with $|A|=\lfloor \frac{\Delta^2}{\log^{3/2} \Delta} \rfloor$ and $|B|=\Delta$.
We consider the random graph obtained by having each vertex of $B$ adjacent to a random subset of $A$ of size $\Delta$. 
The expected codegree of two vertices of $B$ is at most $|A| \cdot \frac{\Delta^2}{|A|^2} \leq \log^{3/2} \Delta$ and of two vertices of $A$ is at most $(\frac{\Delta(\Delta-1)}{|A|(|A|-1)})^2|B|<1$. 
Simple counting shows that the probability that the codegree exceeds $\log^2 \Delta$ is $o(1)$.

For any subset $Z$ of $A$ of size  at most $\frac{\Delta}{\log^{5/4} \Delta}$,
since the function $xe^{-x}$ attains maximum at $x=1$, the probability that a vertex in $B$ is adjacent to exactly one vertex of $Z$ is no more than  
$$\frac{|Z|\Delta}{|A|}(1-\frac{|Z|-1}{|A|})^{\Delta-1}\le \frac{|Z|\Delta}{|A|}e^{-\frac{|Z|-1}{|A|}(\Delta-1)} = \frac{|Z|\Delta}{|A|}e^\frac{-|Z|\Delta}{|A|}e^\frac{|Z|+\Delta-1}{|A|}\le e^{-1}e^\frac{2\log^{3/2}\Delta}{\Delta}.$$

For large enough $\Delta$ this is less than $e^{-1}(1+\frac{2\log^{3/2}\Delta}{\Delta}) = e^{-1} + \frac{2\log^{3/2}\Delta}{e\Delta}$. 
So, the expectation of the random number  $X_Z$ of vertices of $B$ adjacent to exactly one vertex of $Z$ is at most $(e^{-1} + \frac{2\log^{3/2}\Delta}{e\Delta})|B| \leq (e^{-1} + \frac{1}{6\log^{1/9}\Delta})|B|$.  
By the Chernoff Bound,
\begin{align*}
    P\left(X_Z > (e^{-1} + \frac{1}{3\log^{1/9}\Delta})|B|\right) \leq & P\left(|X_z-E[X_z]| >\frac{1}{6\log^{1/9}\Delta}|B|\right) \\
    \leq & e^{-\frac{(\frac{1}{6\log^{1/9}\Delta}|B|)^2}{3E[X_z]}} \\
    \leq & e^{-\frac{(\frac{1}{6\log^{1/9}\Delta}|B|)^2}{3|B|}} \\
    \leq & e^{-\frac{\Delta}{108\log^{2/9}\Delta}}.
\end{align*}

Note that there are at most 
\begin{align*}
    \sum_{i=1}^{\lfloor \frac{\Delta}{\log^{5/4} \Delta} \rfloor}{|A| \choose i} \leq \frac{\Delta}{\log^{5/4} \Delta} \cdot {|A| \choose \lfloor \frac{\Delta}{\log^{5/4} \Delta} \rfloor} \leq & \frac{\Delta}{\log^{5/4} \Delta} \cdot \frac{|A|^{\lfloor \frac{\Delta}{\log^{5/4} \Delta} \rfloor}}{\lfloor \frac{\Delta}{\log^{5/4} \Delta} \rfloor!} \\
    \leq & \frac{\Delta}{\log^{5/4} \Delta} \cdot (\frac{e|A|}{\lfloor \frac{\Delta}{\log^{5/4} \Delta} \rfloor})^{\lfloor \frac{\Delta}{\log^{5/4} \Delta} \rfloor} \\ \leq & \frac{\Delta}{\log^{5/4} \Delta} \cdot (\frac{e\frac{\Delta^2}{\log^{3/2}\Delta}}{\lfloor \frac{\Delta}{\log^{5/4} \Delta} \rfloor})^{\lfloor \frac{\Delta}{\log^{5/4} \Delta} \rfloor} \\
    \leq & (2e\frac{\Delta}{\log^{1/4}\Delta})^{\frac{\Delta}{\log^{5/4} \Delta}} \leq \Delta^{\frac{\Delta}{\log^{5/4} \Delta}} \\
\end{align*}
subsets $S$ of $A$ of size at most $\frac{\Delta}{\log^{5/4} \Delta}$.
So the probability that there is a subset $S$ of $A$ of size at most $\frac{\Delta}{\log^{5/4} \Delta}$ such that there are more than $(e^{-1} + \frac{1}{3\log^{1/9}\Delta})|B|$ vertices in $B$ adjacent to exactly one vertex in $S$ is at most 
$$\Delta^{\frac{\Delta}{\log^{5/4} \Delta}} \cdot e^{-\frac{\Delta}{108\log^{2/9}\Delta}} \leq e^{\frac{\Delta}{\log^{1/4}\Delta} - \frac{\Delta}{108\log^{2/9}\Delta}} = o(1).$$
\end{proof}

For a given $\Delta$, let $G$ be the  graph guaranteed to exist by Lemma \ref{lem_upperbound_bipartite}. 
Let $\phi$ be an arbitrary $(\Delta+1)$-colouring of $G$.
For every $b \in B$, let $c_b$ be the number of colour classes $C$ such that $|N(b) \cap C|=1$.
Let $M= \Sigma_bc_b$. 

There are at most $\frac{|A|\log^{5/4} \Delta}{\Delta} \le  \frac{\Delta}{\log^{1/4} \Delta}$ colour classes containing more than $\frac{\Delta}{\log^{5/4}\Delta}$ vertices in $A$. 
So the colour classes containing more than $\frac{\Delta}{\log^{5/4}\Delta}$ vertices in $A$ contribute at most $|B| \cdot \frac{\Delta}{\log^{1/4} \Delta}$ to $M$ in total.
By Lemma \ref{lem_upperbound_bipartite}, each colour class containing at most $\frac{\Delta}{\log^{5/4}\Delta}$ vertices in $A$ contributes at most $(e^{-1}+\frac{1}{3\log^{1/9}\Delta})|B|$ to $M$.
Therefore, $M \leq |B| \cdot \frac{\Delta}{\log^{1/4} \Delta} + (\Delta+1) \cdot (e^{-1}+\frac{1}{3\log^{1/9}\Delta})|B| \leq (e^{-1}+\frac{2}{3\log^{1/9}\Delta})(\Delta+1)|B|$.
It follows that $c_b \leq (e^{-1}+\frac{2}{3\log^{1/9}\Delta})(\Delta+1)$ for some $b \in B$. 
Since $b$ is adjacent to $\Delta$ vertices in $A$, there are $\deg(b)-c_b \geq \Delta-(e^{-1}+\frac{2}{3\log^{1/9}\Delta})(\Delta+1)>(1-e^{-1}-\frac{1}{\log^{1/9}\Delta})\Delta$ neighbours of $b$ using colours the same as other neighbours of $b$.
So $\phi$ is not $(1-e^{-1}-\frac{1}{\log^{1/9}\Delta})\Delta$-peaceful.
This proves Theorem \ref{upperbound}. 

\section{The Proof of Theorem \ref{lowerbound}} \label{sec:proof_lowerbound}

We prove Theorem \ref{lowerbound} in this section.
We recall that we need only prove the theorem for $\Delta$-regular graphs.
We set $\epsilon=\frac{1}{8001}$.

To motivate our approach we sketch how a proof might look for 
those $G$ with   $\Delta+1-\lceil 20\epsilon \Delta \rceil$ colourings.  
We simply apply the Lovasz Local Lemma to show that we can choose a subset $Z$ of $V(G)$ such that 
every vertex of $G$  
has between   $4 \epsilon \Delta$ and  $\frac{9\epsilon \Delta}{2}$ neighbours in $Z$. 
We next colour $G-Z$ using 
$\Delta+1-\lceil 20\epsilon \Delta \rceil $ colours. We then mimic the proof 
of Theorem \ref{manycolours} to colour $G[Z]$ using a new set  of $\lceil 20\epsilon \Delta \rceil$ colours
so that every vertex $v$  of $G$ of degree at least $\Delta-\epsilon \Delta$
has $\epsilon \Delta$ neighbours in $Z$ whose colour appears
nowhere else on $N(v)$. We would be able to do so   because in a random assignment of the new colours 
to $Z$, we expect  that more than $ \frac{31 \epsilon\Delta}{10}$ neighbours of 
$v$ will  be assigned a colour assigned to no other vertex of  $N(v)$  and that 
at most  $\frac{41\epsilon \Delta}{40}$ neighbours of $v$  will be uncoloured.

Indeed, this approach can be made to  work provided we can choose $Z$ so that $G-Z$ 
has a $\Delta+1-\lceil 20\epsilon \Delta \rceil$  colouring (while still insisting every vertex   has 
between $4 \epsilon \Delta$ and  $\frac{9\epsilon \Delta}{2}$ neighbours in $Z$). 

One obstruction to this approach is the existence of large cliques. 
We would  need to ensure that $Z$ contains at least $k$
vertices of any $\Delta+1-\lceil 20\epsilon\Delta \rceil+k$ clique. But this  could force us to violate
the condition that every vertex has at most $\frac{9 \epsilon \Delta}{2}$ neighbours in $Z$, as $k$ may be as large as $\lceil 20\epsilon \Delta \rceil$  and
hence $Z$ would have to have  minimum degree this large. 

So, we have to modify our approach to take into account the existence of large cliques. 
In fact, large cliques are helpful for producing a peaceful colouring because every vertex in a large clique must have many neighbours using distinct colours, as we saw in the proof of Observation \ref{greedy_obs}. 

We will still find a set $Z$ of vertices such that $G-Z$ has a proper $(\Delta+1-\lceil 20\epsilon \Delta\rceil)$-colouring, colour $G-Z$ using at most this 
many colours and then partially colour $Z$ so that however we complete it, we have an $(1-\epsilon) \Delta$-peaceful colouring.

\begin{definition} \label{useful_Z_def}
  {\rm  A subset $Z$ of $V(G)$ is {\it usable} if it satisfies the following:
  \begin{enumerate} 
  \item[(i)] $G-Z$ has a proper $(\Delta+1-\lceil 20\epsilon\Delta \rceil)$-colouring, 
  \item[(ii)] $Z$ can be partitioned into sets $K_1,...,K_\ell,Y$ (for some nonnegative integer $\ell$), where each $K_i$ is a subset of a clique $C_i$ of size at least $\frac{2\Delta}{3}+1$ with $|K_i|=\lceil 20\epsilon\Delta \rceil$ and the $C_i$ disjoint, 
  \item[(iii)] For every $v$ in a $K_j$, $|N(v) \cap (Z-K_j)| \le \frac{20\epsilon \Delta}{9}$, 
  \item[(iv)] For every $v$ which is in $Y$ or is not in a clique of size exceeding $\frac {2\Delta}{3}+1$, $|N(v) \cap Z| \le \frac{200\epsilon \Delta}{9}$, and $|N(v) \cap Y| \le 2\epsilon \Delta$, 
  \item[(v)]  For every $v \in V(G)$ which is not in a clique of size exceeding $\frac{2\Delta}{3}+1$,  
  we have $\frac{3\epsilon \Delta}{2} \le  |N(v) \cap Z|$. 
  \end{enumerate}}
\end{definition}

Theorem \ref{lowerbound} immediately follows from the combination of the following two lemmas.

\begin{lemma}
\label{Zexists}
  There is a usable subset $Z$ of $V(G)$ 
\end{lemma}

\begin{lemma}
\label{Zisgood}
  If there is  a usable subset $Z$ of $V(G)$ then $G$ is $(1-\epsilon)\Delta$-peaceful.  
\end{lemma}

We will prove these two lemmas in the following two subsections.

\subsection{Proof of Lemma \ref{Zexists}}

Our starting point is a  structural decomposition 
which allows us to partition a graph of maximum degree $\Delta$ 
into a set $S$ of sparse vertices and a set of small dense subgraphs, such that 
every clique of size near $\Delta$ is contained in one of the dense subgraphs. 

We say that a vertex $v$ of $G$ is {\it $d$-sparse} if $|E(G[N(v)])|< {\Delta \choose 2}-d\Delta$; otherwise, $v$ is {\it $d$-dense}.
By a {\it $d$-dense decomposition} of $G$, we mean a partition of $V(G)$ into {\it dense sets} $D_1,...,D_\ell$ (for some nonnegative integer $\ell$) and a set $S$ such that 
\begin{enumerate}
    \item[(a)] every $D_i$ has between $\Delta+1-8d$ and $\Delta+4d$ vertices,
    \item[(b)] a vertex is adjacent to at least $\frac{3\Delta}{4}$ vertices of $D_i$ if and only if it is in $D_i$, and
    \item[(c)] every vertex of $S$ is $d$-sparse.
\end{enumerate}

By Lemma 15.2 of \cite{mr_book}, $G$ has a $d$-dense decomposition for every integer $d \le \frac{\Delta}{100}$. 

We take a $d$-dense decomposition for $d=4\lceil 20 \epsilon \Delta \rceil<\frac{\Delta}{100}$.

We call a colouring of a $D_i$ {\it suitable} if 
    \begin{itemize}
        \item it is a proper colouring of $G[D_i]$ using at most $\Delta+1$ colours, 
        \item the vertices in the colour classes of size 1 form a clique, 
        \item there is no colour class of size exceeding three, and 
        \item if there is a colour class of size three, then the number of colours is $\Delta+1$, and every vertex in a colour class of size three is adjacent to every vertex forming a singleton colour class.
    \end{itemize}

\begin{lemma}
Every $D_i$ has a suitable colouring. 
\end{lemma}

\begin{proof}
We consider a proper colouring of $G[D_i]$ with at most $\Delta+1$ colours minimizing the sum of the squares of the sizes of the colour classes. 

If there is no colour class of size at least three, we can obtain a suitable colouring by repeatedly merging two nonadjacent vertices in singleton colour classes into a colour class of size two until this is no longer possible.

Suppose there is a colour class of size at least three.
Then the singleton colour classes must form a clique, for otherwise we could merge two and then make a vertex in the largest colour class a singleton colour class to decrease the sum of the squares of the size of the colour classes. 

If this colouring uses less than $\Delta+1$ colours, then we can move vertex from a colour class of size at least three to form a singleton colour class and decrease the sum of the squares of the size of the colour classes. 
Hence this colouring uses $\Delta+1$ colours.
Furthermore, all the vertices of a colour class of size at least three must be adjacent to every vertex forming a singleton colour class, for otherwise we could move one of the vertices of the former class to the latter class to decrease the sum of the squares of the size of the colour classes. 

So to show this colouring is suitable, it suffices to show that there is  no colour class of size exceeding three. 
Suppose to the contrary that there is a colour class of size at least four.  
Now, every vertex which is a singleton colour class sees a vertex in every colour class of size two, for otherwise we could merge these two colour classes and make a vertex in the largest colour class a singleton colour class to again contradict our choice of colouring. 
So every vertex in a singleton colour class is adjacent to at least one vertex in each other colour class and is adjacent to all vertices in each colour class with size at least three.
But this is impossible since $G$ has maximum degree $\Delta$ and there are $\Delta+1$ colour classes where one of them has size at least four. 
\end{proof} 

We call a suitable colouring of a $D_i$ {\it very suitable} if it has at least $\Delta-40 \cdot \lceil 20 \epsilon \Delta \rceil$ singleton colour classes.
Let ${\mathcal D} = \{D_i: 1 \leq i \leq \ell, D_i$ has a very suitable colouring$\}$.

\begin{lemma}
\label{suitable=suitable}
For every very suitable colouring of any $D_i$ in ${\cal D}$ and every set $J_i$ consisting of $\lceil 20\epsilon \Delta \rceil$ vertices in the singleton colour classes of the colouring, we have that any proper $(\Delta+1-\lceil 20\epsilon \Delta\rceil)$-colouring of $G[V']$ for a subset $V'$ of $V(G)-D_i$ can be extended to a proper $(\Delta+1-\lceil 20 \epsilon \Delta\rceil)$-colouring of $G[V' \cup (D_i-J_i)]$. 
\end{lemma}

\begin{proof}
The definition of $d$-dense set implies $|D_i| \le \Delta+ 16\lceil 20\epsilon \Delta \rceil$.
Since there are at least $\Delta-40 \cdot \lceil 20 \epsilon \Delta \rceil$ singleton colour classes, there are at most $\frac{|D_i|-(\Delta-40 \cdot \lceil 20 \epsilon \Delta \rceil)}{2} \leq 28 \lceil 20 \epsilon \Delta\rceil <\frac{\Delta}{6}$ colour classes in our colouring which have size larger than one. 
For any colour class $Z$ of size three, each vertex of $Z$ sees every vertex forming a singleton colour class, so $\sum_{v \in Z}|N(v)-D_i| \leq 120 \lceil 20 \epsilon \Delta \rceil<\frac{\Delta}{2} $.
For every colour class $Z$ of size two, the definition of a $d$-dense set implies $\sum_{v \in Z}|N(v)-D_i| \leq 2 \cdot (\Delta-\frac{3\Delta}{4}) \leq \frac{\Delta}{2}$.

Given a proper $(\Delta+1-\lceil 20\epsilon \Delta\rceil)$-colouring of $G[V']$,  we colour the vertices of each nonsingleton colour class of our suitable colouring of $D_i$ greedily, ensuring that we give all the vertices in a given colour class the same colour. 
We can do this because of the remarks of the last paragraph and the fact that $\Delta+1-\lceil 20\epsilon \Delta\rceil > \frac{\Delta}{6} + \frac{\Delta}{2}$. 
Now, we can greedily colour the vertices in a singleton colour class but not in $J_i$.
We can do this as each such vertex sees the $\lceil 20 \epsilon \Delta\rceil$ vertices of $J_i$.
\end{proof}

\begin{lemma}
\label{unsuitable=suitable}
For every $D_i$ not in ${\cal D}$, any proper $(\Delta+1-\lceil 20\epsilon \Delta\rceil)$-colouring of $G[V']$ for a subset $V'$ of $V(G)-D_i$ can be extended to a proper $(\Delta+1-\lceil 20 \epsilon \Delta\rceil)$-colouring of $G[V' \cup D_i]$. 
\end{lemma}

\begin{proof}
Let $\phi$ be a suitable colouring of $D_i$.
The definition of $d$-dense set implies $|D_i| \ge \Delta+1 - 32\lceil 20\epsilon\Delta \rceil$.
Since $D_i \not \in {\mathcal D}$, $\phi$ has at most $\Delta+1 - 40\lceil 20\epsilon\Delta \rceil$ singleton colour classes, so the sum of the sizes of the colour classes of size at least two in $\phi$ is at least $8 \lceil 20 \epsilon \Delta\rceil$.  
Since every colour class of $\phi$ has size at most three, $D_i$ contains at least $\lceil \frac{8 \lceil 20 \epsilon \Delta\rceil}{3} \rceil$ disjoint pairs of nonadjacent vertices. 
Pick $\lceil \frac{8 \lceil 20 \epsilon \Delta\rceil}{3} \rceil$ such pairs $R_1,R_2,...$.
Note that $\lceil \frac{8 \lceil 20 \epsilon \Delta\rceil}{3} \rceil <\frac{\Delta}{150}$.
Since $|D_i| \le \Delta +4d \le \frac{26\Delta}{25}$ and $G[D_i]$ has minimum degree at least $\frac{3\Delta}{4}$, we know that for each $R_j$, there are at least $2 \cdot \frac{3\Delta}{4}-|D_i| \geq \frac{23\Delta}{50}$ vertices in $D_i$ adjacent to both vertices in $R_j$, and at least $\frac{23\Delta}{50}-2\lceil \frac{8 \lceil 20 \epsilon \Delta\rceil}{3} \rceil \geq \frac{67\Delta}{150}$ of them are not in any of those pairs. 

Now we show that there are at least $\frac{\Delta}{30}$ vertices of $D_i$ not in the pairs $R_1,R_2,...$ adjacent to both vertices of $\lceil 20 \epsilon \Delta\rceil$ of the pairs.
Construct a bipartite graph $H$ with a bipartition $\{A,B\}$, where each vertex in $A$ corresponds to a pair $R_j$ and each vertex in $B$ corresponds to a vertex in $D_i$ but not in any pair, such that a vertex $a \in A$ is adjacent to a vertex $b \in B$ in $H$ if and only if $b$ sees both vertices in the pair corresponding to $a$.
So every vertex in $A$ is adjacent in $H$ to at least $\frac{67\Delta}{150}$ vertices in $B$.
Hence $|E(H)| \geq |A| \cdot \frac{67\Delta}{150}$.
Then the vertex in $B$ with the $\lceil \frac{\Delta}{30} \rceil$-th largest degree in $H$ has degree at least $\frac{|E(H)|-|A| \cdot \frac{\Delta}{30}}{|B|} \geq \frac{|A| \cdot \frac{62\Delta}{150}}{|D_i|} \geq \frac{\lceil \frac{8 \lceil 20 \epsilon \Delta\rceil}{3} \rceil \cdot \frac{62\Delta}{150}}{\frac{26\Delta}{25}} > \lceil 20\epsilon\Delta \rceil$, which is desired. 

So there is a set $Z$ consisting of $\lceil \frac{\Delta}{30} \rceil$ vertices of $D_i$ not in the pairs, every element of which sees both vertices of $\lceil 20 \epsilon \Delta\rceil$ of the pairs. 
Since every vertex of $Z$ sees at least $\frac{3\Delta}{4}$ vertices of $D_i$ by the definition of $d$-dense sets, every vertex of $Z$ sees at least $\frac{3\Delta}{4} -2\lceil \frac{8 \lceil 20 \epsilon \Delta\rceil}{3} \rceil \geq \frac{3\Delta}{4}-\frac{\Delta}{75} = \frac{221\Delta}{300}$ vertices of $D_i$ not in the pairs.
So among the vertices in $D_i-Z$ not in the pairs, the one that is adjacent to the $\lceil \frac{\Delta}{3} \rceil$-th largest many vertices in $Z$ is adjacent to at least $\frac{|Z| \cdot \frac{221\Delta}{300} - \frac{\Delta}{3} \cdot |Z|}{|D_i|-|Z|-2\lceil \frac{8 \lceil 20 \epsilon \Delta\rceil}{3} \rceil} \geq \frac{\lceil \frac{\Delta}{30} \rceil \cdot \frac{121\Delta}{300}}{\frac{26\Delta}{25}-\frac{\Delta}{30}} \geq \lceil 20\epsilon\Delta \rceil$ vertices in $Z$.

Hence there exists a set $Y$ of $\lceil \frac{\Delta}{3}\rceil$ vertices of $D_i-Z$ not in the pairs which see at least $\lceil 20 \epsilon \Delta\rceil$ vertices of $Z$.

Finally, we recall that for any pair of vertices of $D_i$, the definition of dense sets implies that the union of the neighbourhoods of the pair outside $D_i$ has at most  $2(\Delta-\frac{3\Delta}{4}) = \frac{\Delta}{2}$ elements. 
So, given a proper $(\Delta+1-\lceil 20\epsilon \Delta\rceil)$-colouring of $G[V']$, we can extend it by colouring our $\lceil \frac{8 \lceil 20 \epsilon \Delta\rceil}{3} \rceil <\frac{\Delta}{150}$ pairs $R_1,R_2,...$ so that the two vertices in each pair receive the same colour. 

Note that every vertex of $D_i-Y-Z$ is adjacent to at least $\frac{3\Delta}{4}-(|D_i|-|Y|-|Z|) \geq \frac{3\Delta}{4}-\frac{26\Delta}{25}+\frac{\Delta}{3}+\frac{\Delta}{30} \geq \lceil 20\epsilon \Delta\rceil$ vertices in $Y \cup Z$.
So we can further extend the proper colouring to $G[V' \cup (D_i-Y-Z)]$. 

Since every vertex of $Y$ sees at least $\lceil 20\epsilon \Delta\rceil$  vertices of $Z$, we can further extend the proper colouring to $G[V' \cup (D_i-Z)]$.
Finally, we extend to $G[V' \cup D_i]$ which we can do because each vertex of $Z$ sees both vertices of $\lceil 20\epsilon \Delta\rceil$ of our pairs, and the two vertices in each pair receive the same colour. 
\end{proof}

We are now ready to prove Lemma \ref{Zexists}. 

For every $D_i \in {\mathcal D}$, take a very suitable colouring $\phi_i$ of $D_i$, and let $C_i$ be the union of the singleton colour classes of $\phi_i$.
The definition of very suitable colouring implies that $C_i$ is a clique of size at least $\Delta-40 \cdot \lceil 20\epsilon\Delta \rceil > \frac{2\Delta}{3}+1$.

We consider the random process where we choose a uniformly random subset $K_i$ of $C_i$ of size $\lceil 20 \epsilon \Delta \rceil$ for each $D_i \in {\cal D}$, and place each vertex in $V(G)-\bigcup_{D_i \in {\mathcal D}} C_i$ into a set $Y$ independently with probability $\frac{7\epsilon}{4}$. 
We choose $K_i$ by taking a random permutation of $C_i$  and letting $K_i$ consist of the first $\lceil 20 \epsilon \Delta \rceil$ elements of the permutation. 
Let $Z = \bigcup_{D_i \in {\mathcal D}}K_i \cup Y$.

\begin{lemma} \label{useful_existence_1}
$Z$ satisfies (i) and (ii) in Definition \ref{useful_Z_def}. 
\end{lemma}

\begin{proof}
Let $G'$ be a graph obtained from $G[S]$ by attaching leaves so that $\Delta(G')=\Delta$.
Since every vertex in $S$ is $d$-sparse, $G'$ is a $160\epsilon$-sparse graph in the sense of \cite{bpp}.
Theorem 1.6 of \cite{bpp} implies that the chromatic number of $G'$ is at most $(1-(0.3012-0.1283) \cdot 160\epsilon)(\Delta+1) \leq \Delta+1-\lceil 20 \epsilon \Delta \rceil$ when $\Delta$ is sufficiently large.
So $G[S]$ is properly $(\Delta+1-\lceil 20 \epsilon \Delta \rceil)$-colourable.
By Lemmas \ref{suitable=suitable} and \ref{unsuitable=suitable}, $G-\bigcup_{D_i \in {\mathcal D}} K_i$ is properly $(\Delta+1-\lceil 20 \epsilon\Delta \rceil)$-colourable. 
Hence $Z$ satisfies (i).

Recall that for every $D_i \in {\mathcal D}$, $C_i$ is a clique of size greater than $\frac{2\Delta}{3}+1$.
So $Z$ satisfies (ii).
\end{proof}

For each $D_i \in {\cal D}$ and $v \in C_i$ we let $A_v$ be the event that $|N(v) \cap (Z-K_i)| > \frac{20\epsilon\Delta}{9}$. 
For each vertex $v$, we let $B_v$ be the event that $|N(v) \cap Z| > \frac{200 \epsilon \Delta}{9}$ and $C_v$ be the event that $|N(v) \cap Y| >2 \epsilon \Delta$.  
For each vertex $v$, we let $E_v$ be the event that $|N(v) \cap Z|\le \frac{3\epsilon \Delta}{2}$.

By Lemma \ref{useful_existence_1}, to prove that $Z$ is useful to finish the proof of this lemma, it suffices to prove that with positive probability, none of those events occur.

Now, each of the events indexed by $v$ depends only on the choices of $K_i$ for $D_i$ in ${\cal D}$
such that $N(v) \cap C_i \neq \emptyset$, and on which of the  neighbours  of $v$ outside these sets were placed in 
$Z$. So they are all mutually independent of all the other events except a set $S_v$ indexed by vertices within 
distance four of $v$. 
It follows that if we can show each of our events has probability at most $\frac{1}{21\Delta^{4}}$, then we can apply the Local Lemma to show that with positive probability none of the events occur and hence there is a usable $Z$. 

Each of our events bounds the size of some set. In every case, to show that the bound holds with 
probability less than $\frac{1}{21\Delta^4}$, we bound the expected size of this set and then apply Corollary \ref{permutation} to bound the difference between the actual size of the set and its expected size. 
In every case, we can certify that the size of the set is $s$ by specifying $s$ choices or permutation positions.
In every case, swapping two elements of the permutation or changing a choice can affect the size of a set by at most 
one.  Thus, in every case, for large enough $\Delta$, the probability that the set has size within  $\Delta^{2/3}$
of its expected size is greater  than $1-\frac{1}{64\Delta^4}$ by Corollary \ref{permutation} since the expected size of those sets are at most $\Delta$. 

We consider first $A_v$. We note as $v\in C_i$ it has at least $|C_i|-1 \geq \Delta -40\lceil 20 \epsilon \Delta \rceil-1$ neighbours in $C_i$, so it has at most $40\lceil 20 \epsilon \Delta \rceil+1$ neighbours in $V(G)-C_i$. 
Each of these neighbours is  either in $Y$ with probability $\frac{7\epsilon}{4}<\frac{1}{360}-10^{-7}$ or is in $\bigcup_{D_i \in {\mathcal D}}K_i$ with probability at most $\frac{\lceil 20\epsilon\Delta \rceil}{\min_j|C_j|} \leq \frac{\lceil 20 \epsilon \Delta \rceil}{\Delta -40\lceil 20 \epsilon \Delta \rceil}<\frac{1}{360}-10^{-7}$.
So each of these neighbours is in $Z$ with probability less than $\frac{1}{360}-10^{-7}$.
Hence the expected size of $N(v) \cap (Z-K_i)$ is less than $(40\lceil 20 \epsilon \Delta \rceil+1) \cdot (\frac{1}{360}-10^{-7})< \frac{20\epsilon\Delta}{9}-\Delta^{2/3}$. 
So we are done. 

We consider next $B_v$. 
Each of the $\Delta$  neighbours of $v$ is in $Z$ with probability at most $\max\{\frac{7\epsilon}{4},\frac{\lceil 20 \epsilon \Delta \rceil}{\Delta -40\lceil 20 \epsilon \Delta \rceil}\} <\frac{1}{360} - 3.8 \cdot 10^{-7} = \frac{8001}{360}\epsilon - 3.8 \cdot 10^{-7} < \frac{200\epsilon}{9}- 3 \cdot 10^{-8}$. 
Hence the expected size of $N(v) \cap Z < \frac{200\epsilon\Delta}{9}-\Delta^{2/3}$ and we are done. 

We consider next $C_v$. Each of the  $\Delta$ neighbours of $v$  is in 
$Y$ with probability at most $\frac{7\epsilon}{4}$.
Hence the expected size of $N(v) \cap Y$ is at most $\frac{7\epsilon \Delta}{4} < 2\epsilon\Delta-\Delta^{2/3}$.
So we can apply our concentration inequality to finish off.

Finally we consider $E_v$. Each of the $\Delta$  neighbours of $v$ is in $Z$ with probability at least $\min\{\frac{7\epsilon}{4},\frac{\lceil 20\epsilon\Delta \rceil}{\max_j|C_j|}\} \geq \min\{\frac{7\epsilon}{4},\frac{\lceil 20\epsilon\Delta \rceil}{\max_j|D_j|}\} \geq \min\{\frac{7\epsilon}{4},\frac{\lceil 20\epsilon\Delta \rceil}{\Delta+4d}\} \geq \frac{7\epsilon}{4}$.
Hence the expected size of $N(v) \cap Z$ is at least $\frac{7\epsilon \Delta}{4} > \frac{3\epsilon\Delta}{2}+\Delta^{2/3}$.
So we can apply our concentration inequality to finish off.
This completes the proof of Lemma \ref{Zexists}. 

\subsection{Proof of Lemma \ref{Zisgood}}

Let $Z$ be a usable subset of $V(G)$ with a partition, as in the definition of usable, into sets $K$ and $Y$, where $K$ is a disjoint union of cliques $K_1,K_2,...,K_\ell$ each contained in a clique in $G$ with size at least $\frac{2\Delta}{3}+1$, for some nonnegative integer $\ell$.

Since $Z$ is usable, we can colour $G-Z$ properly by using $\Delta+1-\lceil 20 \epsilon \Delta \rceil$ colours. 
Then we assign colours from a disjoint set of $c=\lceil 20 \epsilon \Delta \rceil$ colours, say $\{1,...,c\}$, to the vertices of $Z$ as follows. 
For each $K_i$, we choose a random permutation of $K_i$ and assign its $j^{{\rm th}}$ element colour $j$.
We assign every other vertex of $Z$ a uniform random colour from the set $\{1,2,...,c\}$.  
For every monochromatic edge both of whose  endpoints are in $Y$, we uncolour both endpoints. 
For every monochromatic edge with at least one endpoint in $K$, we uncolour only the endpoints in $K$. 

\begin{lemma}
For every vertex $z \in Z$ and colour $r$, the probability that $z$ is assigned  colour $r$  equals $\frac{1}{\lceil 20\epsilon\Delta \rceil}$.
\end{lemma}

\begin{proof}
This follows immediately from the definition and the fact   if $z$ is in some $K_i$, it is equally likely to appear in 
each position of the permutation. 
\end{proof}

Note that the current coloured vertices give a partial proper colouring $\phi$ of $G$ with $\Delta+1$ colours.
For every $v \in V(G)$ which is not in a clique of size at least $\frac{2\Delta}{3}+1$, 
    \begin{itemize}
        \item let $S_v$ be the set of neighbours of $v$ in $Z$ which were assigned a colour not assigned to any other neighbour of $v$  and retained it, 
        \item let $S_v'$ be the set of uncoloured neighbours of $v$ in $Z$, 
        \item let $S_v'' = \{u \in N(v) \cap Z: u$ is assigned  a colour $x$ such that at least $\lceil \log^2\Delta \rceil$ vertices in $N(v) \cap Z$ are assigned $x \}$,
        \item let $X_v = |S_v|-|S_v' \cup S_v''|$, and
        \item let $A_v$ be the event that $X_v<\epsilon \Delta$. 
    \end{itemize}
Note that every vertex in $S_v$ is a $(\phi,v)$-undisturbed vertex.
So if $A_v$ does not occur, then the number of $(\phi,v)$-undisturbed vertices is at least $\epsilon\Delta$ more than the number of uncoloured vertices in $N(v) \cap Z$.

Let $\overline{K}$ be the union of all cliques in $G$ with size at least $\frac{2\Delta}{3}+1$.
Note that $K \subseteq \overline{K}$.
By Observations \ref{greedy_obs}, to prove Lemma \ref{Zisgood}, it suffices to show that with positive probability, none of the events $A_v$ (for $v \in V(G)-\overline{K}$) occur.
Since each $A_v$ depends only on the colour assignments on vertices within distance 2 of $G$,
it suffices to show that $P(A_v) \leq \frac{1}{5\Delta^4}$ for every $v \in V(G)-\overline{K}$ 
as we are then done by the Local Lemma.

Fix $v \in V(G)-\overline{K}$.

We first bound the probability of a vertex being in $S_v'$ or $S_v''$.

\begin{lemma} \label{uncoloured_prob_useful}
For a vertex $x \in Z$, the probability that $x$ is uncoloured is at most  $\frac{1}{9}$.
\end{lemma}

\begin{proof}
We let $K_x$ be $\{x\}$ if $x \in Y$, and let $K_x$ be the $K_i$ containing $x$ otherwise.
We condition on the colour of all the vertices  except  those  in  $K_x$.
Given this conditioning, $x$ receives a uniform random colour. 
If $x \in Y$, then $x$ is uncoloured only when there exists a monochromatic edge between $x$ and a vertex in $Y$, so the conditional and hence unconditioned  probability that $x$ is uncoloured is at most $\frac{|N(x) \cap Y|}{\lceil 20 \epsilon \Delta \rceil}$  which, by the definition of usable set, is at most $\frac{1}{10}$. 
If $v \in K_i$ for some $i$, then the  conditional and hence unconditioned probability that $x$ is uncoloured is at most $\frac{|N(x) \cap (Z-K_i)|}{\lceil 20 \epsilon \Delta \rceil} \leq \frac{20\epsilon\Delta}{9\lceil 20\epsilon\Delta \rceil} \leq \frac{1}{9}$.    
\end{proof}

\begin{lemma} \label{many_neighbour_prob_useful}
For a vertex $x \in N(v) \cap Z$, $P(x \in S_v'') \leq (\frac{23e}{20\log^2\Delta})^{\log^2\Delta}$. 
\end{lemma}

\begin{proof}
If $x \in S_v''$, then there are at least $\lceil \log^2\Delta \rceil$ vertices in $ N(v) \cap Z$ receiving the same colour before uncolouring.
Note that there is no single $K_i$ containing at least two of those $\lceil \log^2\Delta \rceil$ vertices.
So the colours received by those $\lceil \log^2\Delta \rceil$ vertices are independent.
Hence $P(x \in S_v'') \leq {|N(v)  \cap Z| \choose \lceil \log^2\Delta \rceil} \cdot (\frac{1}{\lceil 20\epsilon\Delta \rceil})^{\lceil \log^2\Delta \rceil}$.
Since $v \in V(G)-\overline{K}$, $|N(v)  \cap Z|  \leq 23\epsilon\Delta$.
So 
    \begin{align*}
        P(x \in S_v'') \leq & {\lfloor 23\epsilon\Delta \rfloor \choose \lceil \log^2\Delta \rceil} \cdot (\frac{1}{\lceil 20\epsilon\Delta \rceil})^{\lceil \log^2\Delta \rceil} \\
        \leq & (\frac{23\epsilon\Delta \cdot e}{\lceil \log^2\Delta \rceil})^{\lceil \log^2\Delta \rceil} \cdot (\frac{1}{20\epsilon\Delta})^{\lceil \log^2\Delta \rceil} \leq (\frac{23e}{20\log^2\Delta})^{\log^2\Delta}.
    \end{align*}
\end{proof}

\begin{lemma}
To prove Lemma \ref{Zisgood}, it suffices to show that $E[|S_v|] \geq \epsilon \Delta+\frac{|N(v) \cap Z|}{8}$. 
\end{lemma}

\begin{proof}
Assume $E[|S_v|] \geq \epsilon \Delta+\frac{|N(v) \cap Z|}{8}$.
We shall show that $P(A_v) \leq \frac{1}{5\Delta^4}$.

By Lemmas \ref{uncoloured_prob_useful} and \ref{many_neighbour_prob_useful}, $E[|S_v' \cup S_v''|] \leq (\frac{1}{9}+(\frac{23e}{20\log^2\Delta})^{\log^2\Delta})|N(v) \cap Z| \leq \frac{9}{80}|N(v) \cap Z|$.
Hence $E[X_v] = E[|S_v|]-E[|S_v' \cup S_v''|]  \geq (\epsilon \Delta+\frac{|N(v) \cap Z|}{8}) - \frac{9}{80}|N(v) \cap Z| = \epsilon \Delta+\frac{|N(v) \cap Z|}{80}$.
So $P(A_v) = P(X_v<\epsilon\Delta) \leq P(X_v<E[X_v]-\frac{|N(v) \cap Z|}{80}) \leq P(|X_v-E[X_v]| > \frac{|N(v) \cap Z|}{80})$.

Now, $X_v-E[X_v]=|S_v|-E[|S_v|]-|S'_v \cup S''_v|+E[|S'_v \cup S''_v|]$.
So, $\Big|X_v-E[X_v]\Big| \le \Big||S_v|-E[|S_v|]\Big|+\Big||S'_v \cup S''_v|-E[|S'_v \cup S''_v|]\Big|$.
Furthermore, $\Big||S_v|-E[|S_v|]\Big|=\Big||N(v) \cap Z|-|S_v|-E[|N(v) \cap Z|-|S_v|]\Big|$. 
Finally, letting  $W = \{w \in N(v) \cap Z: w$ is assigned a colour assigned to some other vertex of $(N(v) \cup N(w)) \cap Z\}$,
we see that $W = (N(v) \cap Z)-S_v$. 
Therefore, it suffices to show that for $B \in \{|S'_v \cup S''_v|, |W|\}$,
$P(|B-E[B]| \geq \frac{|N(v) \cap Z|}{160}) < \frac{1}{11\Delta^4}$. 
    
In both cases,  $B$  is a random variables satisfying Theorem \ref{concentrationtheorem} with $c= \lceil \log^2 \Delta \rceil$ and $r=2$. 
In addition  $E[|B|] \leq |N(v) \cap Z|$ and $|N(v) \cap Z| \geq \frac{3\epsilon\Delta}{2}$ by the definition of usable set and our choice of $v$. So, 
$\frac{|N(v) \cap Z|}{320} \geq 6\lceil \log^2 \Delta \rceil \sqrt{4E[|B|]}+32\lceil \log^2 \Delta \rceil^2$, and by Corollary \ref{permutation}, 
    \begin{align*}
        & P\left(\Big||B|-E[|B|]\Big| \geq \frac{|N(v) \cap Z|}{160}\right) \\
        \leq & P\left(\Big||B|-E[|B|]\Big| \geq \frac{|N(v) \cap Z|}{320}+6\lceil \log^2 \Delta \rceil\sqrt{4E[|B|]}+32\lceil \log^2 \Delta \rceil^2\right) \\
        \leq & 2e^{-\frac{(\frac{|N(v) \cap Z|}{320})^2}{16\lceil \log^2 \Delta \rceil^2(E[|B|]+\frac{|N(v) \cap Z|}{320})}} \leq  2e^{-\frac{(\frac{|N(v) \cap Z|}{320})^2}{16\lceil \log^2 \Delta \rceil^22|N(v) \cap Z|}} \leq \frac{1}{15\Delta^4}.
    \end{align*}
This proves the lemma.
\end{proof}

So, we need only show that $E[|S_v|] \geq \epsilon \Delta+\frac{|N(v) \cap Z|}{8}$. 

Let $c(v)$ be the the number of colours assigned to vertices of  $N(v) \cap Z$. 
We see that the number of colours assigned to exactly one neighbour of $v$ is at least $c(v)-(|N(v) \cap Z| -c(v))$. 
So the expected number of such colours is at least $2E[c(v)]-|N(v) \cap Z|$.
Hence by Lemma \ref{uncoloured_prob_useful}, $$E[|S_v|] \geq 2E[c(v)]-|N(v) \cap Z|-E[|S_v'|] \geq 2E[c(v)]-\frac{10|N(v) \cap Z|}{9}.$$ 

So it suffices to show that $E[c(v)] \geq \frac{89|N(v) \cap Z|}{144}+\frac{\epsilon\Delta}{2}$. 
Hence we may assume that $E[c(v)] < \frac{89|N(v) \cap Z|}{144}+\frac{\epsilon\Delta}{2}$.

\begin{lemma} \label{prob_retain_unique_useful}
For $w \in N(v) \cap Z$, $P(w \in S_v) \geq \frac{8}{9}-\frac{E[c(v)]}{\lceil 20 \epsilon \Delta \rceil}$. 
\end{lemma}

\begin{proof}
We first expose the choice of colours on every vertex in $Z$ except $w$ and the vertices in a $K_i$ containing $w$ (if such a $K_i$ exists), then make a choice for $w$. 
So $w$ chooses a uniform colour. 
The  conditional probability that $w$ chooses a colour appearing on other vertices in $N(v) \cap Z$ is at most $\frac{c(v)}{\lceil 20 \epsilon\Delta \rceil}$. 
By Lemma \ref{uncoloured_prob_useful}, the probability that $w$ is uncolored is at most $\frac{1}{9}$.
So $P(w \in S_v) \geq \frac{8}{9}-\frac{E[c(v)]}{\lceil 20 \epsilon\Delta \rceil}$.
\end{proof}

By Lemma \ref{prob_retain_unique_useful}, $E[|S_v|] \geq |N(v) \cap Z| \cdot (\frac{8}{9}-\frac{E[c(v)]}{\lceil 20 \epsilon \Delta \rceil})$, which exceeds $|N(v) \cap Z| \cdot (\frac{8}{9}-\frac{\frac{89|N(v) \cap Z|}{144}+\frac{\epsilon\Delta}{2}}{\lceil 20 \epsilon \Delta \rceil})$.

\begin{lemma}
    $$\epsilon \Delta+\frac{|N(v) \cap Z|}{8}  \leq |N(v) \cap Z| \cdot (\frac{8}{9}-\frac{\frac{89|N(v) \cap Z|}{144}+\frac{\epsilon\Delta}{2}}{\lceil 20 \epsilon \Delta \rceil}).$$
\end{lemma}

\begin{proof}
Suppose to the contrary that $$\epsilon \Delta+\frac{|N(v) \cap Z|}{8}  >|N(v) \cap Z| \cdot (\frac{8}{9}-\frac{\frac{89|N(v) \cap Z|}{144}+\frac{\epsilon\Delta}{2}}{\lceil 20 \epsilon \Delta \rceil}).$$
This is equivalent to: 
$$\lceil 20 \epsilon \Delta \rceil(\epsilon \Delta -\frac{55|N(v) \cap Z|}{72}) > -|N(v) \cap Z|(\frac{89|N(v) \cap Z|}{144}+\frac{\epsilon\Delta}{2}).$$
By the definition of a usable set, $\frac{3\epsilon \Delta}{2} \leq |N(v) \cap Z|$, so $\epsilon \Delta -\frac{55|N(v) \cap Z|}{72} <0$.
Hence 
$$20 \epsilon \Delta(\epsilon \Delta -\frac{55|N(v) \cap Z|}{72}) > -|N(v) \cap Z|(\frac{89|N(v) \cap Z|}{144}+\frac{\epsilon\Delta}{2}).$$
I.e.
$$\frac{89|N(v) \cap Z|^2}{144} - \frac{133\epsilon\Delta}{9}|N(v) \cap Z| + 20 (\epsilon \Delta)^2>0.$$
That is, 
$$\frac{89}{144}(\frac{|N(v) \cap Z|}{\epsilon\Delta})^2 - \frac{133}{9}\frac{|N(v) \cap Z|}{\epsilon\Delta} + 20>0.$$
The quadratic polynomial $f(x) = \frac{89}{144}x^2 - \frac{133}{9}x + 20$ satisfies $f(0)>0$, $f(\frac{3}{2})<0$ and $f(\frac{200}{9})<0$, so $f(x)<0$ for every $\frac{3}{2} \leq x \leq \frac{200}{9}$.
Since $\frac{3\epsilon \Delta}{2} \leq |N(v) \cap Z| \leq \frac{200\epsilon \Delta}{9}$ by the definition of usable set, the left hand side of the above inequality is negative, a contradiction.
\end{proof}

Therefore, $E[|S_v|] \geq |N(v) \cap Z| \cdot (\frac{8}{9}-\frac{\frac{89|N(v) \cap Z|}{144}+\frac{\epsilon\Delta}{2}}{\lceil 20 \epsilon \Delta \rceil}) \geq \epsilon \Delta+\frac{|N(v) \cap Z|}{8}$ which completes the proof of Lemma \ref{Zisgood}.

\section{Proof of Theorem \ref{sparsegraphs}} \label{sec:proof_sparsegraphs}

In this section, we consider peaceful colourings of  graphs of maximum degree $\Delta$ and 
maximum codegree $\frac{\sqrt{\Delta}}{\log^8 \Delta}$. We recall that to prove Theorem \ref{sparsegraphs}, we can 
restrict our attention to $\Delta$-regular graphs. We do so. 

We consider an iterative procedure (similar to that found in Chapter 12 of \cite{mr_book}). We  activate  
a $o(\frac{1}{\log \Delta})$ proportion of the vertices in each iteration,  and then assign them a colour not yet assigned to any 
of their neighbours. We continue until we have coloured all but at most $\frac{\Delta}{\log \log \Delta}$ neighbours 
of each vertex. We let $i^*$ be the number of iterations we perform. 

At the end of the process we may  need to uncolour some of the neighbours of each vertex  because of conflicts with adjacent  vertices coloured in the same iteration, but only a $o(\frac{1}{\log \Delta})$ proportion 
which does not significantly change the number of colours appearing exactly once in the neighbouhood.  

We need to focus on four variables:   
    \begin{itemize}
        \item $L^v_i$ is the list of available colours which have not yet been assigned to any neighbour of $v$ at the start of iteration $i$.
        \item $Good^v_i$ is the set of colours which have been assigned to exactly one neighbour of $v$. 
        \item $n^v_i$ is the number of neighbours of $v$ coloured in the $i^{th}$ iteration.
        \item $D^v_i$ is the number of uncoloured neighbours of $v$ at the start of the $i^{th}$ iteration, 
    \end{itemize}
so $D^v_i=\Delta-\sum_{j=1}^{i-1} n^v_j$. 

We compare  this process with an idealized  process performed on a $\Delta$-star.
We activate each vertex with probability $ \alpha$ in each iteration 
and assign  each activated vertex a uniformly random colour. 
We let $L^*_i$ be the list of colours not yet assigned to a leaf of the star  at the start of the $i^{th}$ 
iteration. We define $Good^*_i$, $n^*_i$ and $D^*_i$  analogously.  

Now, the activation procedure does not affect which colours are assigned to each vertex.
We will obtain the same colouring of the vertices if we colour all the vertices we colour all at once. 
So, if $i$ is large enough that we have coloured all but  $ o(\frac{\Delta}{\log \log \Delta})$ vertices of the star then $E[|Good^*_{i+1}|]$ is within $ o(\frac{\Delta}{\log \log \Delta})$ of the expected number of vertices which appear on exactly 
one leaf of the star if we assign them independent uniform colours. 
Now, letting $c=\Delta+\frac{B \Delta}{\log \log \Delta}$.  we observe that  the probability $p$  that in doing so, we  colour a vertex using a colour not used 
on any other vertex satisfies\footnote{The last inequality is a consequence of Taylor's approximation to $e$.}:      $$p= (1-\frac{1}{c})^{\Delta-1} \ge   (1-\frac{1}{c})^{c} \ge  \frac{1}{e}(1-o(\frac{1}{\sqrt{c}})). $$

So, the expected number of such vertices is at least $\frac{\Delta}{e}-\sqrt{\Delta}$.  
We will ensure that for every $v$ and $i \le i^*$,  $|Good^v_i|-|Good^*_i|=o(\frac{\Delta}{\log \log \Delta})$ which allows us to use this observation to prove the theorem. 

To make this approach work, we must show that in each iteration the probability we assign a colour to a neighbour of 
$v$ is essentially the same for every colour, except that used to colour $v$ which will never be assigned to a 
neighbour. In order to do so,  we need to keep track of the set $U^{v,c}_i$ which consists of those vertices $w$ in $N(v)$ which have not been assigned a colour and  for which $c$ is in $L^w_i$. 
That is,
    \begin{itemize}
        \item $U^{v,c}_i = \{w \in N(v): w$ has not been assigned a colour and $c \in L^w_i\}$. 
    \end{itemize}
We need to show that for any $c$ and $c'$ (both not assigned to $v$) we have $|U^{v,c}_i| \approxeq |U^{v,c'}_i|$. 

It is this part of the proof that requires a bound on the codegree, allowing us to show that for every 
$c$, $|U^{v,c}_i|-|U^{v,c'}_{i+1}|$ is highly concentrated around its expected value. This is not true 
in graphs with large codegree. If $G$ is $K_{\Delta,\Delta}$ then every $|U^{v,c}_i|$ is either 0 or the 
number of uncoloured neighbours of $v$, and it  can drop from the latter to the former in one iteration.

This is the broad outline of our approach (We have omitted one important technical wrinkle). Forthwith the details.

\subsection{Postprocessing}

Alon, Krivelevich, and Sudakov \cite{aks} have shown that graphs of maximum degree $\Delta$ and maximum codegree 
$\frac{\Delta}{f}$ are $O(\frac{\Delta}{\log f})$-colourable. Applying their result we obtain that  as long as we choose $\Delta_0$ 
large enough there is a $C'$ such that 
$G$  has a $\lceil \frac{C' \Delta}{\log \Delta} \rceil$-colouring.  
We can apply this result in a postprocessing step. Thus, the main part of our proof is the following.

\begin{theorem}
\label{sparsegraphs2}
     There is a $\Delta_0$ and a  $B$     
     such that for $\Delta \ge \Delta_0$, every graph of  maximum degree $\Delta$ and codegree at most $\frac{\sqrt{\Delta}}{\log^8 \Delta}$ has a proper partial $(\Delta+\lceil \frac{B\Delta}{\log \log  \Delta} \rceil)$-colouring 
      such that every vertex $v$  has  at least $e^{-1}\Delta-\frac{B\Delta}{\log\log\Delta}$
      neighbours coloured with a colour assigned to no other vertex of $N(v)$.   
\end{theorem}

We  set $C=3B$ and apply Theorem \ref{sparsegraphs2} to $G$.
We uncolour those vertices coloured with colours in $\{1,2,...,\lceil \frac{C' \Delta}{\log \Delta} \rceil\} \cup
\{\Delta+2,..., \Delta+\lceil \frac{B}{\log \log \Delta} \rceil\}$.
We use \cite{aks} to colour all uncoloured vertices using $\{1,2,...,\lceil \frac{C' \Delta}{\log \Delta} \rceil\}$. 
We note that for every vertex
$v$, there are at least $e^{-1}\Delta-\frac{2B\Delta}{\log\log\Delta}-\lceil \frac{C' \Delta}{\log \Delta} \rceil$
 neighbours of $v$ coloured with a colour assigned to no other vertex of $N(v)$.  
The proof of Theorem \ref{sparsegraphs} is complete.  

So it remains to prove Theorem \ref{sparsegraphs2}.

\subsection{Analyzing The Idealized Process}

We set $\alpha=\frac{1}{\log^2\Delta}$ and $i^*=\alpha^{-1} \log \log \Delta$. 
We consider our idealized random process on a star with $\Delta$ leaves. 

We set $g_1=0$ and 
$l_1=\Delta+\lceil \frac{B \Delta}{\log  \log \Delta} \rceil$.
We define $D_1=\Delta$, and 
$n_1= \alpha D_1$. 
For $i \geq 1$, we set :
\begin{itemize}
\item$l_{i+1}=l_i(1-\frac{1}{l_1})^{n_i}$,
\item $g_{i+1}=g_i(1-\frac{1}{l_1})^{n_i}+\frac{n_il_i}{l_1}$,
\item $D_{i+1}=D_i-n_i$,
\item $n_{i+1}= \alpha D_{i+1}$. 
\end{itemize}

We note $\sum_i n_i \le D_1$, $l_i \ge l_1(1-\frac{1}{l_1})^{\sum_i n_i} \ge l_1(1-\frac{\sum_i n_i}{l_1-1}) \ge
l_1-D_1 - \frac{D_1}{l_1-1} \ge \frac{0.9B\Delta}{\log \log \Delta}$. 
We prove:

\begin{lemma} \label{idealized}
    With probability $1-o(\Delta^{-4})$
    for every $i$ from  $1$ to $i^*+1$ we have:
    \begin{enumerate}
       \item[(I)] $|D^*_i-D_i| \le i\Delta^{2/3}$.
        \item[(II)] $|n^*_i-n_i| \le i\Delta^{2/3}$.
      \item[(III)] $\Big||L^*_i|-l_i\Big| \le i\Delta^{3/4}$.
       \item[(IV)] $\Big||Good^*_i|-g_i\Big| \le \frac{6i\Delta}{\log^4 \Delta}$.
        \end{enumerate}
\end{lemma}

\begin{proof}
Clearly, (I), (III) and (IV) hold for $i=1$.
We prove for each $i$ that given that (I),(III) and (IV) hold for $i$, with probability $o(\Delta^{-5})$ (II) holds for $i$.
The expectation of $n^*_i$ is $\alpha D^*_{i}$. So
$E[n^*_i]-n_{i}=\alpha(D^*_i-D_i)$. Given (I) holds for 
$i$,  
$$|E[n^*_{i}]-n_{i}| \le  \alpha i \Delta^{2/3}.$$
Since $n^*_i$ is determined by the independent activation of the at most $\Delta$ neighbours of $v$
and each can change $n^*_i$ by at most one, applying Azuma's inequality  completes the first part of the proof.  

We then prove that given (I), (II), (III) and (IV) hold for $i$, with probability $o(\Delta^{-5})$ (I), (III) and (IV) hold for $i+1$. 

The expectation of $D^*_{i+1}$ is $(1-\alpha)|D^*_{i}|$. So
given (I) holds for $i$,  
$$|E[D^*_{i+1}]-D_{i}| \le  \alpha i \Delta^{2/3}.$$

Now, $E[|L^*_{i+1}|]=|L^*_{i}|(1-\frac{1}{l_1})^{n^*_i}$.
Letting $l'_{i+1}=|L^*_{i}|(1-\frac{1}{l_1})^{n_i}$,
we see that given that (III) holds for $i$, $|l'_{i+1}-l_{i+1}| \le  i\Delta^{3/4}$.
Furthermore, given that (II) holds for $i$, since $l_1 \ge l'_{i+1}$: 
$$\Big|E[|L^*_{i+1}|]-l'_{i+1}\Big| = \Big|l'_{i+1}(1-(1-\frac{1}{l_1})^{n_i^*-n_i})\Big| \le \Big|l'_{i+1} (1-(1-\frac{1}{l_1})^{-i\Delta^{2/3}})\Big| \le 2i\Delta^{2/3}.$$
So, 
$$|E[L^*_{i+1}]-l_{i+1}| \le \frac{(2i+1)\Delta^{3/4}}{2}.$$

The  expectation of $|Good^*_{i+1}|$ lies between   
$$|Good^*_i|(1-\frac{1}{l_1})^{n^*_i}+\frac{n^*_i|L^*_i|}{l_1} \text{  and  } |Good^*_i|(1-\frac{1}{l_1})^{n_i^*}+\frac{n_i^*|L^*_i|}{l_1}  -\frac{1.1(n^*_i)(n^*_i -1)|L^*_i|}{(l_1)^2}.$$

Now, since (II) holds for $i$, $n^*_i \le  2 \alpha \Delta$, so since $|L^*_i| \le l_1$,
for large enough $\Delta$, 
$$\frac{1.1(n^*_i)(n^*_i-1)|L^*_i|}{(l_1)^2} \le \frac{4.4\Delta}{\log^4\Delta}.$$

Now, since (II) and (III) hold for $i$, 
$$|\frac{n^*_i|L^*_i|}{l_1}-\frac{n_il_i}{l_1}| \le \frac{i\Delta^{3/4}(n_i+l_i+ i\Delta^{2/3})}{l_1}\le 3i\Delta^{3/4}.$$

Letting $g'_{i+1}=|Good^*_{i}|(1-\frac{1}{l_1})^{n_i} + \frac{n_il_i}{l_1}$,
we see that given that (III) holds for $i$, $$|g'_{i+1}-g_{i+1}| \leq (1-\frac{1}{l_1})^{n_i}|Good^*_i-g_i| \le \frac{6 i\Delta}{\log^4 \Delta}.$$
Furthermore, since $l_1 \ge g'_{i+1}$: 
$$\Big|(|Good^*_{i}|(1-\frac{1}{l_1})^{n^*_i} +\frac{n^*_i|L^*_i|}{l_1})-g'_{i+1}\Big| \le \Big|g'_{i+1} (1-(1-\frac{1}{l_1})^{-i\Delta^{2/3} })\Big| + \Big|\frac{n^*_i|L^*_i|}{l_1} - \frac{n_il_i}{l_1}\Big| \le 2i\Delta^{2/3} + 3i\Delta^{3/4}.$$
So, 
$$\Big|E[|Good^*_{i+1}|]-g_{i+1}\Big| \leq \frac{1.1(n^*_i)(n^*_i-1)|L^*_i|}{(l_1)^2} + (2i\Delta^{2/3} + 3i\Delta^{3/4}) + \frac{6 i\Delta}{\log^4 \Delta} \le \frac{(6i+5)\Delta}{\log^4\Delta}.$$ 
Since all our variables depend on the activation and colour assignment choices to the neighbours of $v$, and each activation and colour assignment choice can affect 
a variable by at most one, three applications of Azuma's Inequality, yield the second part of the proof. 
\end{proof}

We note that (IV) of the lemma implies that $\Big|E[|Good_i^*|]-g_i\Big| \le \frac{6i\Delta}{\log^4 \Delta}+ o(\frac{\Delta}{\Delta^4}) \le \frac{7i\Delta}{\log^4 \Delta}$ since $\Big||Good_i^*|-g_i\Big| \leq \Delta$.

\subsection{The Iterative Procedure}

In this section we present the details of the  quasi-random iterative process we use to complete the 
conflict-free colouring. 
Note that we can assume the graph $G$ is $\Delta$-regular since we can embed every graph of maximum degree at most $\Delta$ into a $\Delta$-regular graph.

We have already defined  $L^v_i$, $Good^v_i$, $U^{v,c}_i$,$D^v_i$, $n^v_i$, 
$g_i,m_i,l_i,D_i$ and $n_i$.  
We shall also keep track of 
    \begin{itemize}
        \item $Bad^v_i$ which is the set of vertices in $N(v)$ contained in a monochromatic edge at the start of the $i$th iteration. 
    \end{itemize}
This counts vertices of $G$ that we have to remove at the end of the process in order to obtain a proper partial colouring.

We will show\footnote{ Lemma \ref{mainlemma} (E) below}  that with positive probability we can ensure 
that for every $v$ and $i$, $|Good^v_i-g_i|$ is $o(\frac{\Delta}{\log \log \Delta})$. 
This implies that $|Good^v_{i^*+1}-e^{-1}\Delta|=o(\frac{\Delta}{\log \log \Delta})$ by Lemma \ref{idealized}. 
We shall also show\footnote{ Lemma \ref{mainlemma} (B) below} that $Bad^v_{i^*+1}=o(\frac{\Delta}{\log \log \Delta})$ for every $v$. 
These two facts complete the proof of Theorem \ref{sparsegraphs2} since we can take the colouring at the start of the $(i^*+1)$-th iteration and removing the colours on the vertices in $\bigcup_{v \in V(G)}Bad^v_{i^*+1}$.  

We note that $L^v_1$ is the full set of colours and has size $\Delta + \lceil \frac{B\Delta}{\log\log\Delta} \rceil$ while $Good^v_1$ and  $Bad^v_1$ are all empty.
Both  $U^{v,c}_1$ and $D^v_i$ are  the neighbourhood of $v$ and  have  size $\Delta$. 

In each of $i^*$ iterations, we activate each vertex to which we have not yet assigned 
a colour with probability $\alpha$. 
We then assign each  activated $v$ a colour  uniformly chosen from $L^v_i$.

The  technical wrinkle mentioned earlier is that we now remove some colours from the lists to 
make the procedure easier to analyze. The random process above may remove 
some colours $c$ from $L^v_i$ with a higher probability than others because $v$ has more neighbours
with the colour $c$ on their list. So, for every vertex $v$ and colour $c$, we will perform 
an equalizing coin flip to ensure that every colour is removed from every list with the
same probability. Then,  we remove a uniformly chosen random subset of each $L^v_i$
so that every $L^v_{i+1}$ has the same size.  

We will show that we can do this so that for every $v$,  $Good^v_i$ stays near 
$g_i$, and $L^v_i$ stays near $l_i$. Much of the argument 
is similar to that given in the last section, determining expected values and showing Random variables are concetrated around them with high probability.  We will also need to apply the Lov\'{a}sz Local 
Lemma to show that we can ensure appropriate behaviour in all neighbourhoods. However this is not much more 
work because we know the probability of the bad events are tiny.

One key difference is that we need to show that the probability two different colours are assigned to a neighbour  of $v$ in an iteration  are very close to equal. 
It is for this reason that we track $U^{v,c}_i$ for every $v$ and $c$. The second key difference is that we need to bound $|Bad^v_i|$. 

Letting  $l'_1=l_1$ and $l'_i=l_i- \lceil \frac{i \Delta}{\log^5 \Delta} \rceil$, we  will show: 

\begin{lemma}
For large enough $\Delta$, we can carry out our procedure such that the following hold
for every $i$ between  $1$ and $i^*+1$. 
\label{mainlemma} 
\begin{enumerate}
\item[(A)] for every vertex $v$, $D_i-\Delta^{2/3} \le D^v_i \le D_i+\Delta^{2/3}$.
\item[(B)] for every vertex $v$, at most $\frac{\Delta}{\log^3 \Delta}$ neighbours of $v$ assigned a colour in the iteration $i-1$ are in a monochromatic edge with another vertex assigned the same colour in the iteration $i-1$. Hence
$Bad_{i^*}^v \le \frac{i^*\Delta}{\log^3 \Delta}=o(\frac{\Delta}{\log \log \Delta})$ 
\item[(C)] for every vertex $v$ and colours $c,c'$ not assigned to $v$: 
$\Big||U^{v,c}_i|-|U^{v,c'}_i|\Big| \le \frac{\Delta}{\log^4 \Delta}$.
\item[(D)] for every vertex $v$, $|L^v_i|=l'_i$.
\item[(E)] for every vertex $v$, $|Good^v_i| \geq g_i- \frac{i \Delta}{\log^{7/2} \Delta}$. 
\end{enumerate}
\end{lemma} 

 We note that  (A) and our choice of $i^*$ imply that all but $o(\frac{\Delta}{\log \log 
 \Delta})$  neighbours of $v$ have been assigned a colour by the end of iteration $i^*$. 
 Furthermore, (B) implies that $Bad^v_{i^*+1} \leq (i^*+1)\frac{\Delta}{\log^3 \Delta} = o(\frac{\Delta}{\log \log \Delta})$ which combined 
 with (E) for $i=i^*+1$ and our estimate of $g_i$ from the last section proves the theorem. So, it remains to prove Lemma \ref{mainlemma}. 
 
\begin{proof}
We note that (A) to (E) trivially hold at the start of the first iteration. 
So we need only show that  conditioned on  (A)-(E) holding for all $v$ for all  $j \le i$,   
the probability they hold for  all $v$ and $i+1$  is positive.

We do this in two steps, we first show that given the conditioning, with positive probability 
we can ensure that before  making the choices removing colours from the lists so that the $L^v_{i+1}$
all have size $l'_i$, we can ensure (A), (B) and (E) for $i+1$ and the following two conditions hold: 
\begin{enumerate}
\item[(C')] for every $v$  and $c,c'$ not assigned to $v$: 
$\Big||U^{v,c}_i|-|U^{v,c'}_i|\Big| \le \frac{\Delta}{\log^4 \Delta}-\Delta^{3/4}$,
\item[(D')]  for every $v$, $l'_{i+1} \le |L^v_{i,1}| \le l'_{i+1}+2\Delta^{2/3}$, 
\end{enumerate}
where $L^v_{i,1}$ is the list of colours for vertex $v$ right after we perform an equalizing coin flip and right before we remove colours from $L^v_{i,1}$ so that the lists $L^v_{i+1}$ have size $l'_{i+1}$.

In the second step, for each $i$, we choose a uniform random subset of $L^v_{i,1}$ with $|L^v_{i,1}|-l'_{i+1}$ vertices to delete from $L^v_{i,1}$ and obtain $L^v_{i+1}$. 
Since the probability a fixed colour is deleted from $L^v_{i,1}$ is at most $\frac{2\Delta^{2/3}}{|L^v_{i,1}|} \leq \frac{\Delta^{2/3}}{l'_{i+1}}=o(\Delta^{-1/4})$, an application of Lemma \ref{permutation} with $c=r=1$, shows that with positive probability we can make these 
choices so that fewer than $\frac{|U^{v,c}_i|}{2\Delta^{1/4}} \leq \frac{\Delta^{3/4}}{2}$ vertices in any $U^{v,c}_{i}$ choose $c$. 
Since (C') held after the first step, doing so ensures (C) holds at the end of the second. 

So it remains to prove only that we can carry out the first step as claimed. 
We will show that given our conditioning, the probability (A),(B),(C'),(D'), and (E)  hold for  a specific $v$ and $i+1$ 
is $o(\Delta^{-5})$. Since this event depends only on the colours assigned to the first and second neighbourhoods 
of $v$, an application of the Local Lovasz Lemma proves the lemma.

We let $X_v$ be the event that (X) does not hold for  $v$ and $i+1$ for $X \in \{A,B,C',D',E\}$. 
We need only show that given our conditioning,  the probability of each $X_v$ is $o(\Delta^{-5})$. 

We consider first $A_v$. 
Because (A) holds for $i$, and $E[D^v_{i+1}]=(1-\alpha)D^v_i$,
we know that $|E[D^v_{i+1}]-D_{i+1}| \le (1- \alpha)\Delta^{2/3}$. 
We are done by an application of Azuma's Inequality.

Before bounding the probability of the remaining events, we need to to make some preliminary observations.

We let $l^*(v,i)$ be $l_1$ if $v$ is not yet coloured and $l_1-1$ if $v$ is coloured.  
Note that $l^*(v,i)$ equals the number of colours $c$ such that $U^{v,c}_i \neq \emptyset$.
The sum of $|U^{v,c}_i|$ over all colours $c$ is $D^v_il'_i$ so the average size of a $U^{v,c}_i$
is $\frac{D^v_il'_i}{l^*(v,i)}$. 
Since (C) and (D) hold for $v$ and $i$, $l_1-1 \ge \Delta \ge D_i$ and $l'_i \le l_1$, it follows that $$\frac{D^v_il'_i}{l_1}-\frac{\Delta}{\log^4 \Delta}-1 \leq |U^{v,c}_i| \leq \frac{D^v_il'_i}{l_1}+\frac{\Delta}{\log^4 \Delta}+1.$$  
Since (D) holds for all vertices and $i$, letting $q(v,c,i)$ be the probability that the colour $c$ is not assigned to a neighbour of $v$ during iteration $i$, we have for every colour $c$ not assigned to $v$:     

$$q(v,c,i)=(1-\frac{\alpha}{l'_i})^{|U^{v,c}_i|}=  (1-\frac{\alpha}{l'_i})^\frac{D_il'_i}{l_1} (1-\frac{\alpha}{l'_i})^{|U^{v,c}_i|-\frac{D_il'_i}{l_1}}.$$

We apply the following: 

\begin{lemma} \label{estimate_qSc2}
    $$(1-\frac{1}{\Delta})(1-\frac{1}{l_1})^{n_i} \leq (1-\frac{\alpha}{l'_i})^\frac{D_il'_i}{l_1} \leq (1+\frac{1}{\Delta})(1-\frac{1}{l_1})^{n_i}.$$
\end{lemma}

\begin{proof}
By Taylor's approximation, for every sufficiently large   $x$, $1-\frac{1}{x} \leq e^{-1/x} \leq 1-\frac{1}{x+1}=(1-\frac{1}{x})\frac{x^2}{x^2-1}$, so $(1-\frac{1}{x^2})e^{-1/x} \leq 1-\frac{1}{x} \leq e^{-1/x}$.
So $$(1-(\frac{\alpha}{l'_i})^2)e^{-\frac{\alpha}{l'_i}} \leq 1-\frac{\alpha}{l'_i} \leq e^{-\frac{\alpha}{l'_i}}.$$
Hence $$e^{-\frac{D_i\alpha}{l_1}} \geq (1-\frac{\alpha}{l'_i})^\frac{D_i l'_i}{l_1} \geq (1-(\frac{\alpha}{l'_i})^2)^\frac{D_il'_i}{l_1}e^{-\frac{D_i\alpha}{l_1}} \geq (1-(\frac{\alpha^2 D_i }{(l'_i)^2}))e^{-\frac{\alpha^2 D_i}{(l_i')^2}}e^{-\frac{D_i\alpha}{l_1}} \geq (1-\frac{2\alpha^2 D_i}{(l_i')^2})e^{-\frac{D_i\alpha}{l_1}}.$$
Since $l_i' \ge l_i -\frac{i\Delta}{\log^5 \Delta}$, $l_i \ge  \frac{B \Delta}{ \log \Delta}$, $D_i \le \Delta$, and $n_i=\alpha D_i$, we have:
$$e^{-\frac{n_i}{l_1}} \geq (1-\frac{\alpha}{l_j'})^\frac{D_il'_i}{l_1} \geq (1-\frac{1}{2\Delta}) e^{-\frac{n_i}{l_1}}.$$
Similarly, since $n_i \le \alpha \Delta$,
$$e^{-\frac{n_j}{l_1}} \geq (1-\frac{1}{l_1})^{n_j} \geq (1-\frac{1}{\ell_1^2})^{n_j}e^{-\frac{n_j}{l_1}} \geq (1-\frac{2n_j}{l_1^2})e^{-\frac{n_j}{l_1}} \geq (1-\frac{1}{2\Delta}) e^{-\frac{n_j}{l_1}}.$$
Therefore, 
$$(1+\frac{1}{\Delta})(1-\frac{1}{l_1})^{n_j} \geq (1+\frac{1}{\Delta})(1-\frac{1}{2\Delta})e^{-\frac{n_j}{l_1}} \geq e^{-\frac{n_j}{l_1}} \geq (1-\frac{\alpha}{l_j'})^\frac{D_il'_i}{l_1} \geq (1-\frac{1}{2\Delta}) e^{-\frac{n_i}{l_1}} \geq (1-\frac{1}{\Delta})(1-\frac{1}{l_1})^{n_j}.$$
\end{proof}

\begin{lemma} \label{estimate_qSc3}
    For every vertex $v$ and colour $c$ which was not assigned to $v$ in a previous iteration, $$|q(v,c,i) - (1-\frac{1}{l_1})^{n_i}| = |(1-\frac{\alpha}{l'_i})^{|U^{v,c}_i|} - (1-\frac{1}{l_1})^{n_i}| \leq \frac{1}{\log^{11/2}\Delta}.$$
\end{lemma}

\begin{proof}
Recall $$\frac{D^v_il'_i}{l_1}-\frac{\Delta}{\log^4 \Delta}-1 \leq |U^{v,c}_i| \leq \frac{D^v_il'_i}{l_1}+\frac{\Delta}{\log^4 \Delta}+1.$$
Since (A) holds $i$, we have $$-\frac{2\Delta}{\log^4 \Delta} \leq (\frac{D^v_il'_i}{l_1}-\frac{D_il'_i}{l_1})-\frac{\Delta}{\log^4 \Delta}-1 \leq |U^{v,c}_i|-\frac{D_il'_i}{l_1} \leq (\frac{D^v_il'_i}{l_1}-\frac{D_il'_i}{l_1})+\frac{\Delta}{\log^4 \Delta}+1 \leq \frac{2\Delta}{\log^4 \Delta}.$$
By Lemma \ref{estimate_qSc2} and the remarks preceding it we have:
$$(1-\frac{\alpha}{l'_i})^{\frac{2\Delta}{\log^4 \Delta}}(1-\frac{1}{\Delta})(1-\frac{1}{l_1})^{n_i} \le     q(v,c,i)  \le  (1-\frac{\alpha}{l'_i})^{-\frac{2\Delta}{\log^4 \Delta}}  (1+\frac{1}{\Delta})(1-\frac{1}{l_1})^{n_i}.$$
Since $l_i' \ge l_i -\frac{i\Delta}{\log^5 \Delta}$, $l_i \ge  \frac{B \Delta}{\log  \log \Delta}$, and $\alpha=\log^{-2} \Delta$,
the result follows. 
\end{proof}

Let $$p^*_i=\frac{l'_{i+1}+\Delta^{2/3}}{l'_i}.$$
We note that 
    \begin{align*}
        l'_i \cdot q(v,c,i) \geq & l'_i((1-\frac{1}{l_1})^{n_i}-\frac{1}{\log^{11/2} \Delta}) \\
        \ge & (l_i-\frac{i\Delta}{\log^5\Delta})((1-\frac{1}{l_1})^{n_i}-\frac{1}{\log^{11/2} \Delta}) \ge l_{i+1}-\frac{((2i+1)\Delta}{2\log^5 \Delta}>l'_{i+1}+\Delta^{2/3}.
    \end{align*}
So there exists $0<x(v,c,i)<1$ such that $q(v,c,i) \cdot (1-x(v,c,i)) = \frac{l'_{i+1}+\Delta^{2/3}}{l'_i} = p^*_i$.

Now we perform our equalizing coin flips: for every vertex $v$ and colour $c \in L^v_i$, we remove $c$ from $L^v_i$ with probability $x(v,c,i)$.
Note that the coin flips are independent of the activation and colour assignments.
So for every vertex $v$ and colour $c \in L^v_i$, the overall probability that $c$ is retained in $L^v_i$ equals $q(v,c,i) \cdot (1-x(v,c,i)) = p^*_i$.  
Note that 
    \begin{align*}
        p^*_i= \frac{l'_{i+1}+\Delta^{2/3}}{l'_i} > (1-\frac{1}{l_1})^{n_i}-\frac{\Delta}{l'_i\log^5 \Delta} > & (1-\frac{1}{l^2_1})^{n_i}(1-\frac{n_i}{l_1})-\frac{1}{\log^{5}\Delta} \\
        \geq & 1-\frac{n_i}{l_1} - \frac{2n_i}{l_1^2}-\frac{1}{\log^{5}\Delta} >1-\alpha-\frac{1}{\log^{5}\Delta}.
    \end{align*}

Now, the set $L^v_i-L^v_{i,1}$ of colours removed from $L^v_i$ is determined by the choice of the  colours assigned to the vertices of $N(v)$ and the result of the  equalizing coin flips for $v$. 
Each choice or result can change the number of these colours by at most one. 
So since the expected number of colours removed is at most $l_1<2\Delta$, applying Lemma \ref{permutation} with $c=r=1$, we see that with probability $1-o(\Delta^{-5})$, $\Big||L^v_{i,1}|-E[|L^v_{i,1}|]\Big| = \Big||L^v_{i,1}|-p^*_i l_i'\Big| \leq \Delta^{2/3}$, and hence between $l'_{i+1}$ and $l'_{i+1}+2\Delta^{2/3}$. 
That is, the probability that (D') is violated for $v$ is $o(\Delta^{-5})$.

We consider next $B_v$. Lemma \ref{estimate_qSc3} tells us that 
 $$q(v,c,i) \ge (1-\frac{1}{l_1})^{n_i}-\log^{-11/2} \Delta > 1-\alpha -\log^{-11/2} \Delta >1-2\alpha.$$ 
Thus, the probability that some neighbour of a vertex is assigned  a specific colour $c$ is at most $2\alpha$
and the probability that a vertex is in a monochromatic edge is at most $4\alpha^2$.
It  follows that the expectation of the number  $Y_v$ of  vertices in $N(v)$ which are in a monochromatic edge in the $i^{th}$ iteration  is at most  $\frac{4\Delta}{\log^4 \Delta}$. Rather than proving concentration of $Y_v$, we consider $Y'_v$ which consists of the union of $Y_v$ and those vertices in $N(v)$ coloured with a colour appearing on more than $\log^2 \Delta$ neighbours of $v$. Now the expected number of sets of $s$ neighbours of 
$v$ coloured with the same colour is at most $\Delta^s(\frac{\alpha}{l'_i})^s \le (\frac{B \log \log \Delta}{\log^2 \Delta})^s$. Thus, $E[|Y'_v-Y_v|]=o(1)$. 
Furthermore, the size of $Y'_v$ is  determined by the choices of activation and colours assignments of vertices, and each such choice can change $Y'_v$ by at most $\lceil \log^2 \Delta \rceil$. 
Thus we can apply Lemma \ref{permutation} with $r=2$ and $c=\lceil \log^2 \Delta \rceil$ to show that the probability $B_v$ holds is $o(\Delta^{-5})$. 

We turn to $C'_v$. For each $c$ and $v$, the expected size of $U^{v,c}_{i+1}$ at the end of the 
first step is $|U^{v,c}_i|(1-\alpha)p^*_i$. So  since (C) holds for every $v$ for $i$, 
for any $c$ and $c'$, at the end of the first step 
$$\Big|E[|U^{v,c}_{i+1}|]-E[|U^{v,c'}_{i+1}|]\Big| =\Big||U^{v,c}_i|-|U^{v,c'}_i|\Big|(1-\alpha) p^*_i\le \frac{\Delta}{\log^4 \Delta}(1-\alpha).$$
Hence to show that $C'_v$ holds  with probability $o(\Delta^{-5})$ it is enough to show that 
for every colour $c$, $$P\left(\Big||U^{v,c}_{i+1}|-E[|U^{v,c}_{i+1}|]\Big| \ge \frac{\Delta}{3\log^6\Delta}\right)=o(\Delta^{-6}).$$
Or equivalently: $$P\left(\Big||U^{v,c}_i-U^{v,c}_{i+1}|-E[|U^{v,c}_i-U^{v,c}_{i+1}|]\Big| \ge \frac{\Delta}{3\log^6\Delta}\right)=o(\Delta^{-6}).$$

Now, $1-p^*_i$ is less than $2\alpha$,
so $E[|U^{v,c}_i-U^{v,c}_{i+1}|] = |U^{v,c}|(1-(1-\alpha)p^*_i) \le 3\alpha \Delta =\frac{3\Delta}{\log^2 \Delta}$.
Furthermore, $|U^{v,c}_i-U^{v,c}_{i+1}|$ is determined by the equalizing coin flips and the choice of colours assigned to 
neighbours of neighbours of $v$. By our codegree condition, each choice can change this random variable by at 
most $\frac{\sqrt{\Delta}}{\log^8 \Delta}$. So, applying  Lemma \ref{permutation} with $c=\frac{\sqrt{\Delta}}{\log^8 \Delta}$and $r=1$  and $t=\frac{\Delta}{6 \log^6 \Delta}$ we  obtain the desired result.

We turn now to $E_v$. 
Let $Zero^v_i$ be the set of colours not used on $N(v)$ at the start of iteration $i$.
Note that $L^v_i \subseteq Zero^v_i$.
So $|Zero^v_i| \geq |L^v_i| = l_i'$.
Let $Two^v_i$ be the number of colours in $Zero^v_i$ that are assigned to at least two neighbours of $v$ in the $i^{th}$ iteration.
Then
$$E[|Good^v_{i+1}|]=\sum_{c \in Good^v_{i}} q(v,c,i) + \sum_{c \in Zero^v_i}(1-q(v,c,i))- E[Two^v_i].$$
Applying Lemma \ref{estimate_qSc3}, we see that $$E[|Good^v_{i+1}|] \geq |Good^v_i|((1-\frac{1}{l_1})^{n_i} - \frac{1}{\log^{11/2}\Delta}) + |Zero^v_i|(1-(1-\frac{1}{l_1})^{n_i} - \frac{1}{\log^{11/2}\Delta}) - E[Two^v_i].$$
We note that $E[Two^v_i] \le  |Zero^v_i|\frac{\alpha^2 \Delta^2}{(l'_i)^2} \leq |Zero^v_i|(\frac{\alpha \Delta}{\frac{0.9B\Delta}{\log\log\Delta} - \lceil \frac{i\Delta}{\log^5\Delta} \rceil})^2 \leq |Zero^v_i|\frac{(\log\log\Delta)^2}{\log^4\Delta}$.
So $$E[|Good^v_{i+1}|] \geq |Good^v_i|((1-\frac{1}{l_1})^{n_i} - \frac{1}{\log^{11/2}\Delta}) + |Zero^v_i|(1-(1-\frac{1}{l_1})^{n_i} - \frac{1}{\log^{11/2}\Delta}-\frac{(\log\log\Delta)^2}{\log^4\Delta}).$$
Note that $1-(1-\frac{1}{l_1})^{n_i} \geq 1 - e^{-n_i/l_1} \geq 1-(1-\frac{n_i}{l_1}+(\frac{n_i}{l_1})^2) = \frac{n_i}{l_1} - (\frac{n_i}{l_1})^2 \geq \frac{1}{2\log^3\Delta}$.  
So $1-(1-\frac{1}{l_1})^{n_i} - \frac{1}{\log^{11/2}\Delta}-\frac{(\log\log\Delta)^2}{\log^4\Delta} \geq 0$.
Hence, since $|Zero^v_i| \geq l_i'$, $$E[|Good^v_{i+1}|] \geq |Good^v_i|((1-\frac{1}{l_1})^{n_i} - \frac{1}{\log^{11/2}\Delta}) + l_i'(\frac{n_i}{l_1} - (\frac{n_i}{l_1})^2 - \frac{1}{\log^{11/2}\Delta}-\frac{(\log\log\Delta)^2}{\log^4\Delta}).$$
Since (E) holds for $i$,
    \begin{align*}
        & E[|Good^v_{i+1}|] \\
        \geq & (g_i-\frac{i\Delta}{\log^{7/2}\Delta})(1-\frac{1}{l_1})^{n_i} - \frac{\Delta}{\log^{11/2}\Delta} +l_i'(\frac{n_i}{l_1} - (\frac{n_i}{l_1})^2 - \frac{1}{\log^{11/2}\Delta}-\frac{(\log\log\Delta)^2}{\log^4\Delta}) \\
        = & (g_{i+1}-\frac{n_il_i}{l_1}) - \frac{i\Delta}{\log^{7/2}\Delta}(1-\frac{1}{l_1})^{n_i} - \frac{\Delta}{\log^{11/2}\Delta} +l_i'(\frac{n_i}{l_1} - (\frac{n_i}{l_1})^2 - \frac{1}{\log^{11/2}\Delta}-\frac{(\log\log\Delta)^2}{\log^4\Delta}) \\
        \geq & (g_{i+1}-\frac{n_i(l_i-l_i')}{l_1}) - \frac{i\Delta}{\log^{7/2}\Delta} - \frac{\Delta}{\log^{11/2}\Delta} - \frac{(n_1)^2}{l_1} - \frac{1}{\log^{11/2}\Delta}-\frac{(\log\log\Delta)^2}{\log^4\Delta} \\
        \geq & g_{i+1}-\frac{n_i}{l_1}\frac{i\Delta}{\log^5\Delta} - \frac{i\Delta}{\log^{7/2}\Delta} -\frac{2\Delta}{\log^4\Delta} \\
        \geq & g_{i+1} - \frac{(i+0.5)\Delta}{\log^{7/2}\Delta}.
    \end{align*}

Thus the event $E_v$ is contained in the event that $\Big||Good^v_{i+1}|-E[|Good^v_{i+1}|]\Big| > \frac{\Delta}{2\log^{7/2} \Delta}$.
Now $|Good^v_{i+1}|$ is determined by the choices of colour assignments on the at most $\Delta$ neighbours of 
$v$, and each choice can effect the variable by at most one. So, an application of Azuma's inequality shows that 
$P(E_v)=o(\Delta^{-5})$. 
\end{proof}

\bigskip

\bigskip

\noindent{\bf Acknowledgement:} 
This work was partially conducted when the first author visited the Institute of Mathematics of Academia Sinica,  Taiwan. He thanks the Institute of Mathematics of Academia Sinica for its hospitality.

Both authors thank a referee for a helpful report.

\end{document}